\documentclass[francais]{smfart}

\usepackage[latin1]{inputenc}
\usepackage[francais,english]{babel}
\usepackage{amsmath} 
\usepackage{amssymb}
\usepackage[all,2cell,emtex]{xy}
\input xypic            

\SelectTips{cm}{10}
\SilentMatrices

\theoremstyle{plain}
\newtheorem{thm}{Th\'eor\`eme}[section]
\newtheorem{prop}[thm]{Proposition}
\newtheorem{lemme}[thm]{Lemme}
\newtheorem{cor}[thm]{Corollaire}
\theoremstyle{remark}
\newtheorem{rem}[thm]{Remarque}
\newtheorem{ex}[thm]{Exemple}
\theoremstyle{definition}
\newtheorem{df}[thm]{D\'efinition}
\newtheorem{paragr}[thm]{}
\newtheorem{subparagr}{}[thm]

\numberwithin{equation}{thm}

\newcommand{\titparagr}[1]{\emph{\textbf{#1.\ }}}
\newcommand{\titsubparagr}[1]{\emph{\textbf{#1.\ }}}

\newenvironment{GMlargeitemize}{\begin{list}{}{\setlength{\itemsep}{1.8pt plus .6pt minus .6pt}\setlength{\topsep}{3.5pt plus .8pt minus .8pt}\setlength{\labelwidth}{20pt}\setlength{\leftmargin}{32pt}}}{\end{list}}

\newenvironment{GMitemize}{\begin{list}{}{\setlength{\itemsep}{1.8pt plus .6pt minus .6pt}\setlength{\topsep}{3.5pt plus .8pt minus .8pt}\setlength{\labelwidth}{10pt}\setlength{\leftmargin}{22pt}}}{\end{list}}

\newcommand{\nbd}{\nobreakdash}

\newcommand{\cf}{\emph{cf.}}

\newcommand{\gmremarque}[1]{\relax}

\newcommand{\smsp}{\ }

\newcommand{\W}{\mathcal{W}}
\newcommand{\Winf}{\W_{\infty}}
\newcommand{\Wn}[1]{\W_{#1}}
\newcommand{\Wzer}{\W_{0}}
\newcommand{\Wtr}{\W_{\mathrm tr}}
\newcommand{\Wgr}{\W_{\mathrm gr}}
\newcommand{\Wmod}[1]{\W_{#1}}
\newcommand{\Wder}[1]{\W_{#1}}

\newcommand{\C}{\mathcal{C}}
\newcommand{\D}{\mathbb{D}}
\newcommand{\Dermod}[1]{\D_{#1}}
\newcommand{\Cat}{{\mathcal{C} \mspace{-2.mu} \it{at}}}
\newcommand{\CAT}{\mathcal{CAT}}
\newcommand{\Ens}{{\mathcal{E} \mspace{-2.mu} \it{ns}}}
\newcommand{\pref}[1]{{\widehat{ #1 }}}
\newcommand{\op}[1]{{#1}^{\circ}}
\newcommand{\id}[1]{1^{}_{#1}}

\newcommand{\Fl}{\operatorname{\mathsf{Fl}}}
\newcommand{\Ob}{\operatorname{\mathsf{Ob}}}
\newcommand{\Homint}{\operatorname{\underline{\mathsf{Hom}}}}
\newcommand{\Hom}{\operatorname{\mathsf{Hom}}}

\newcommand{\Q}{\mathcal{Q}}

\newcommand{\comma}{\mathop{\downarrow}}
\newcommand{\circh}{\circ^{}_{\mathsf h}}
\newcommand{\circv}{\circ^{}_{\mathsf v}}

\newcommand{\holim}{\mathop{\oalign{{\rm holim}\cr
\hidewidth$\mathrel{\hbox to 9.2mm{\leftarrowfill}}$\hidewidth\cr}}}
\newcommand{\hocolim}{\mathop{\oalign{{\rm holim}\cr
\hidewidth$\mathrel{\hbox to 9.2mm{\rightarrowfill}}$\hidewidth\cr}}}

\newcommand{\cm}[2]{\mathchoice {#1\raise -1.8pt\vbox{\hbox{$\kern -.8pt/#2$}}} {#1\raise -1.8pt\vbox{\hbox{$\kern -.8pt/#2$}}\kern .8pt} {#1\raise -1.8pt\vbox{\hbox{$\scriptstyle\kern -.8pt /#2$}}} {#1\raise -1.8pt\vbox{\hbox{$\scriptscriptstyle\kern -.8pt /#2$}}} }

\newcommand{\mc}[2]{\mathchoice {\raise -1.8pt\vbox{\hbox{$#1\backslash$}}#2} {\raise -1.8pt\vbox{\hbox{$#1\backslash$}}#2} {\raise -1.8pt\vbox{\hbox{$\scriptstyle#1\backslash$}}#2} {\raise -1.8pt\vbox{\hbox{$\scriptscriptstyle#1\backslash$}}#2} }

\newcommand{\Toto}[2]{{\hskip -2.5pt\xymatrixcolsep{#1pc}\xymatrix{\ar[r]^{#2}&}\hskip -2.5pt} }
\newcommand{\Otot}[2]{{\hskip -2.5pt\xymatrixcolsep{#1pc}\xymatrix{&\ar[l]_{#2}}\hskip -2.5pt} }
\newcommand{\toto}{{\mathchoice{\Toto{1.3}{}}{\Toto{.7}{}}{{\scriptstyle\kern 1.5pt\to\kern 1.5pt}}{{\scriptscriptstyle\kern .5pt\to\kern .5pt}}}}
\newcommand{\otot}{\Otot{.7}{}}
\newcommand{\mapstoto}{{\mathchoice{{\hskip -2.5pt\xymatrixcolsep{1.3pc}\xymatrix{\ar@{|->}[r]&}\hskip -2.5pt}}{{\hskip -2.5pt\xymatrixcolsep{.7pc}\xymatrix{\ar@{|->}[r]&}\hskip -2.5pt}}{{\scriptstyle\kern 1.5pt\mapsto\kern 1.5pt}}{{\scriptscriptstyle\kern .5pt\mapsto\kern .5pt}}}}

\author{Georges Maltsiniotis}
\address{Institut de Math\'ematiques de Jussieu\\
Universit\'e Paris 7 Denis Diderot\\
Case Postale 7012\\
B\^atiment Chevaleret\\
75205 PARIS Cedex 13\\
FRANCE}
\email{maltsin\at math.jussieu.fr}
\urladdr{http://www.math.jussieu.fr/\raise -3.3pt\vbox{\hbox{$\widetilde{ \ }\,$}}maltsin/}

\title[Carrés exacts homotopiques, et dérivateurs]{Carrés exacts homotopiques,\\ et dérivateurs}
\alttitle{Homotopical exact squares and derivators}

\begin{document}

\frontmatter

\begin{abstract}
Le but de ce texte est d'introduire une variante homotopique de la notion de carré exact, étudiée par René Guitart, et d'expliquer le rapport de cette généralisation avec la théorie des dérivateurs.
\end{abstract}

\begin{altabstract}
The aim of this paper is to generalize  in a homotopical framework the notion of exact square introduced by René Guitart, and explain the relationship between this generalization and the theory of derivators.
\end{altabstract}


\maketitle
\tableofcontents
\newpage

\mainmatter

\section*{Introduction}

Cet article est consacré à une généralisation homotopique de la notion de carré exact de Guitart~\cite{Guit1,Guit2,GuitSplit,Guit3,Guit4,Guit5}. Si $A$ est une petite catégorie, on note $\pref{A}$ la catégorie des préfaisceaux d'ensembles sur $A$, et si $u:A\toto B$ est un foncteur entre petites catégories, on note $u^*$ le foncteur image inverse par $u$. 
\[
u^*:\pref{B}\toto\pref{A}\smsp,\qquad F\mapstoto Fu
\]
Le foncteur $u^*$ admet un adjoint à gauche $u^{}_!:\pref{A}\toto\pref{B}$, et un adjoint à droite \hbox{$u^{}_*:\pref{A}\toto\pref{B}$}. Pour tout \emph{carré} dans la 2\nbd-catégorie des petites catégories
\[
\mathcal{D}\ =\quad
\raise 23pt
\vbox{
\UseTwocells
\xymatrixcolsep{2.5pc}
\xymatrixrowsep{2.4pc}
\xymatrix{
A'\ar[d]_{u'}\ar[r]^{v}
\drtwocell<\omit>{\alpha}
&A\ar[d]^{u}
\\
B'\ar[r]_{w}
&B
&\hskip -30pt,
}
}
\]
consistant en la donnée de quatre petites catégories $A$, $B$, $A'$, $B'$, de quatre foncteurs \hbox{$u:A\toto B$}, $u':A'\toto B'$, $v:A'\toto A$, $w:B'\toto B$, et d'un morphisme de foncteurs $\alpha:uv\toto wu'$, on a des morphismes de \og changement de base\fg
\[
c^{}_{\mathcal{D}}:w^*u^{}_*\toto u'_*v^*\qquad\hbox{et}\qquad c'_{\mathcal{D}}:v^{}_!u'^*\toto u^*w^{}_!\smsp,
\]
transposés l'un de l'autre. Le carré $\mathcal{D}$ est \emph{exact au sens de Guitart} si l'un de ses morphismes de foncteurs (donc les deux) est un isomorphisme.
\smallbreak

Pour toute catégorie $\C$, et toute petite catégorie $A$, on note $\C(A)$ la catégorie des préfaisceaux sur $A$ à valeurs dans $\C$ (catégorie des foncteurs de la catégorie opposée $\op{A}$ de $A$, vers $\C$), de sorte que si $\Ens$ désigne la catégorie des ensembles, alors $\pref{A}=\Ens(A)$. Un morphisme $u:A\toto B$ de $\Cat$ (la catégorie des petites catégories) définit un foncteur image inverse $u^*_{\C}:\C(B)\toto C(A)$, noté plus simplement $u^*$, quand aucune confusion n'en résulte. Si la catégorie $\C$ est complète (resp. cocomplète), alors le foncteur $u^*$ admet un adjoint à droite $u^{}_*=u^{\C}_*:\C(A)\toto\C(B)$ (resp. un adjoint à gauche $u^{}_!=u^{\C}_!:\C(A)\toto\C(B)$), et si le carré $\mathcal{D}$ ci-dessus est exact au sens de Guitart, le morphisme de changement de base 
\[
c^{\C}_{\mathcal{D}}:w^*_{\C}\circ u^{\C}_*\toto u'^{\C}_*\circ v^*_{\C}\qquad\hbox{(resp.}\quad c'^{\C}_{\mathcal{D}}:v^{\C}_!\circ u'^*_{\C}\toto u^*_{\C}\circ w^{\C}_!\ )\ 
\]
est un isomorphisme.
\smallbreak

Soit $\C$ une catégorie de modèles de Quillen~\cite{Qu0} complète et cocomplète. Pour toute petite catégorie $A$, on note $\D_{\C}(A)$, ou plus simplement $\D(A)$, la catégorie homotopique de $\C(A)$, localisation de $\C(A)$ par les équivalences faibles argument par argument. Pour tout morphisme $u:A\toto B$ de $\Cat$, le foncteur $u^*=u^*_{\C}:\C(B)\toto\C(A)$ respecte les équivalences faibles argument par argument, et induit donc par localisation un foncteur $u^*=u^*_{\D}:\D(B)\toto\D(A)$. Ce foncteur admet un adjoint à droite \hbox{$u^{}_*=u^{\D}_*:\D(A)\toto\D(B)$}, et un adjoint à gauche $u^{}_!=u^{\D}_!:\D(A)\toto\D(B)$~\cite{CiDer},~\cite{DHKS}. On dit que le carré $\mathcal{D}$ ci-dessus est \emph{homotopiquement exact} (pour la catégorie de modèles~$\C$) si les morphismes de changement de base (transposés l'un de l'autre)
\[
c^{\D}_{\mathcal{D}}:w^*_{\D}\circ u^{\D}_*\toto u'^{\D}_*\circ v^*_{\D}\qquad\hbox{et}\qquad c'^{\D}_{\mathcal{D}}:v^{\D}_!\circ u'^*_{\D}\toto u^*_{\D}\circ w^{\D}_!\ 
\]
sont des isomorphismes. On démontre que cette notion dépend assez peu de la catégorie de modèles $\C$. De façon plus précise, elle ne dépend que de la classe $\Wmod{\C}$ des flèches $u:A\toto B$ de $\Cat$ induisant, pour tout objet $X$ de $\C$, un isomorphisme $\hocolim_AX\toto\hocolim_BX$ de la colimite homotopique du foncteur constant de valeur $X$ indexé par $A$, vers celle de celui indexé par $B$. La classe $\Wmod{C}$ satisfait les propriétés de ce que Grothendieck appelle un \emph{localisateur fondamental}~\cite{PS},~\cite{Der},~\cite{Ast}. Ainsi, les localisateurs fondamentaux fournissent un cadre naturel pour définir la notion de carré exact homotopique.
\smallbreak

Le foncteur $\D=\D_{\C}$ qu'on vient d'associer à une catégorie de modèles de Quillen complète et cocomplète  $\C$
\[
A\mapstoto\D(A)\smsp,\qquad u\mapstoto u^*_{\D}
\]
s'étend facilement aux transformations naturelles, définissant ainsi un dérivateur de Grothendieck~\cite{CiDer}. Un \emph{dérivateur} est un $2$\nbd-foncteur contravariant de la $2$\nbd-catégorie des petites catégories vers celle des catégories (non nécessairement petites), satisfaisant une liste d'axiomes~\cite{PS},~\cite{Der},~\cite{Mal1}. Le concept de dérivateur a été introduit par Grothendieck comme l'objet principal de l'algèbre homotopique, les catégories de modèles jouant le même rôle vis-à-vis des dérivateurs que les présentations par générateurs et relations vis-à-vis des groupes. Des notions proches de celle de dérivateur ont été étudiées par Heller~\cite{Hel,Hel2,Hel3,Hel4}, Keller~\cite{Bernhard} et Franke~\cite{Franke}. Les principales applications de la théorie des carrés exacts homotopiques se situent dans la problématique des dérivateurs, où les propriétés de changement de base, analogues à celles des morphismes propres ou lisses en géométrie algébrique~\cite[exposés~12, 13 et 16]{SGA4}, jouent un rôle capital. Un des buts du présent article est de montrer l'enchaînement du formalisme des dérivateurs avec celui des carrés exacts homotopiques.
\smallbreak

Dans la première section, on rappelle la définition d'un localisateur fondamental et des nombreuses notions qui lui sont attachées, et on présente les principaux exemples de localisateurs fondamentaux. Dans la deuxième section, on associe, à chaque localisateur fondamental, une notion de carré exact, et on explique comment retrouver le cas particulier des carrés exacts de Guitart. On démontre, de façon élémentaire, les propriétés les plus importantes des carrés exacts. En particulier, on obtient que pour tout localisateur fondamental, les carrés comma ainsi que les carrés de Beck-Chevalley sont exacts. En revanche, contrairement au cas des carrés exacts de Guitart, les carrés cocomma ne sont \emph{pas} en général exacts. On étudie le rapport de la notion de carré exact avec celle de foncteur propre ou lisse, introduite par Grothendieck~\cite{PS},~\cite{Der},~\cite{Ast},~\cite{Mal2}. Enfin, on introduit une très légère variante de la notion de carré exact, celle de carré exact faible, qui apparaît naturellement dans la théorie des dérivateurs.
\smallbreak

La troisième section est consacrée aux structures définies sur la $2$\nbd-catégorie des petites catégories par la donnée d'une classe de carrés, qu'on appellera carrés exacts, satisfaisant à divers propriétés de stabilité et de non trivialité ou \og d'initialisation\fg{}. On caractérise la classe des carrés exacts homotopiques comme la plus petite classe de carrés satisfaisant à \emph{toutes} ces propriétés, et on obtient, en particulier, une caractérisation analogue de la classe des carrés exacts de Guitart. Par ailleurs, on démontre, à l'aide de manipulations \og géométriques\fg{} de carrés, des lemmes utiles à la théorie des dérivateurs.
\smallbreak

Le but de la dernière section est de présenter la théorie des dérivateurs du point de vue des carrés exacts homotopiques. On montre que la plupart des résultats élémentaires sur les dérivateurs, démontrés par Grothendieck dans~\cite{Der}, sont conséquences des propriétés formelles des structures de carrés exacts, et des propriétés des carrés exacts homotopiques. Inversement, ces propriétés prennent tout leur sens sous l'éclairage des dérivateurs. On s'applique à isoler soigneusement ceux parmi les axiomes des dérivateurs utiles pour chaque énoncé, ce qui s'avère crucial pour les applications, vu que souvent dans les exemples une partie seulement de ces axiomes est satisfaite.

\section{Rappels sur les localisateurs fondamentaux\\ et les notions qui en découlent}

\begin{paragr} \titparagr{La définition des localisateurs fondamentaux}
On note $\Cat$ la catégorie des petites catégories, et $e$ la catégorie ponctuelle, objet final de $\Cat$. On dit qu'une partie $\W$ de $\Fl(\Cat)$
est {\it faiblement satur\'ee} si elle satisfait aux conditions suivantes:
\begin{GMlargeitemize}
\item[FS1] Les identit\'es sont dans $\W$.
\item[FS2] Si deux des trois flèches d'un triangle commutatif sont dans $\W$, il en est de même de la troisième.
\item[FS3] Si $i:A'\toto A$ et $r:A\toto A'$ sont deux morphismes de $\Cat$ tels que $ri=\id{A'}$, et si $ir$ est dans $\W$, il en est de même de $r$.
\end{GMlargeitemize}
On rappelle qu'un \emph{localisateur fondamental}~\cite{PS},~\cite{Der},~\cite{Ast} est une classe $\W$ de flèches de $\Cat$ satisfaisant aux conditions suivantes:
\begin{GMlargeitemize}
\item[LA] La partie $\W$ de $\Fl(\Cat)$ est faiblement saturée.
\item[LB] Si $A$ est une petite catégorie admettant un objet final, alors $A\toto e$ est dans~$\W$.
\item[LC] Si
$$
\xymatrixcolsep{1pc}
\xymatrix{
A\ar[rr]^{u}\ar[dr]_{v}
&&B\ar[dl]^{w}
\\
&C
}$$
est un triangle commutatif de $\Cat$, et si pour tout objet $c$ de $C$,
le foncteur $\cm{u}{c}:\cm{A}{c}\toto\cm{B}{c}$ induit par $u$ est dans $\W$, alors $u$ est dans $\W$.
\end{GMlargeitemize}
On dit que la classe des flèches $\W$ est un \emph{localisateur fondamental faible} si elle satisfait aux conditions LA et LB, et à la condition LC seulement pour $C=B$ et $w=\id{B}$, autrement dit à la condition:
\begin{GMlargeitemize}
\item[LC${}_{\mathrm f}$] Si $u:A\toto B$ est un morphisme de $\Cat$, et si pour tout objet $b$ de $B$, le foncteur $\cm{A}{b}\toto\cm{B}{b}$ induit par $u$ est dans $\W$, alors $u$ est dans $\W$.
\end{GMlargeitemize}
Les éléments de $\W$ s'appellent des \emph{$\W$\nobreakdash-équivalences}, ou \emph{équivalences faibles}, quand aucune ambiguïté n'en résulte.

\end{paragr}

\begin{paragr} \titparagr{Exemples de localisateurs fondamentaux} Il existe de nombreux exemples de localisateurs fondamentaux.

\begin{subparagr} \titsubparagr{\boldmath Le localisateur fondamental $\Winf$} \label{defWinf}
Le localisateur fondamental le plus important est celui des équivalences faibles usuelles de $\Cat$, classe des foncteurs entre petites catégories dont l'image par le foncteur nerf est une équivalence faible simpliciale, autrement dit, un morphisme d'ensembles simpliciaux dont la réalisation topologique est une équivalence d'homotopie. Il sera noté $\Winf$, et ses éléments seront souvent appelés des $\infty$\nobreakdash-\emph{équivalences}. Le fait que $\Winf$ est un localisateur fondamental \emph{faible} résulte directement du théorème~A de Quillen~\cite{Qu}, qui n'est autre que la condition LC${}_{\mathrm f}$. La preuve de Quillen s'adapte facilement pour montrer la condition plus forte LC, qui est une version \emph{relative} du théorème~A~\cite[théorème 2.1.13]{Ci2}, et Cisinski démontre que $\Winf$ est le plus petit localisateur fondamental~\cite[théorème 2.2.11]{Ci2}, et même le plus petit localisateur \hbox{faible~\cite[théorème~6.1.18]{CiAst}}. 
\end{subparagr}

\begin{subparagr} \titsubparagr{\boldmath Les localisateurs fondamentaux $\Wn{n}$, $n\geqslant0$} \label{defWn}
On rappelle que pour un entier $n\geqslant0$, une $n$\nbd-équivalence d'espaces topologiques est une application continue induisant une bijection des $\pi_0$, et un isomorphisme des $i$\nbd-èmes groupes d'homotopie pour tout $i$, $1\leqslant i\leqslant n$, et tout choix de point base. On dit qu'un foncteur entre petites catégories est une $n$\nbd-\emph{équivalence} si la réalisation topologique de son nerf est une $n$\nbd-équivalence topologique. Les $n$\nbd-équivalences de $\Cat$ forment un localisateur fondamental~\cite[section 9.2]{CiAst}, noté~$\Wn{n}$. Le localisateur fondamental $\Winf$ est l'intersection des $\Wn{n}$, $n\geqslant0$.
\end{subparagr}

\begin{subparagr} \titsubparagr{\boldmath Les localisateurs fondamentaux $\Wtr$ et $\Wgr$} \label{defWgr}
La classe de toutes les flèches de $\Cat$ est un localisateur fondamental, qu'on appelle \emph{trivial}. Les foncteurs entre petites catégories 
toutes deux non vides, ou toutes deux vides,
forment un localisateur fondamental, qu'on appelle \emph{grossier}. On démontre que les seuls localisateurs fondamentaux qui ne sont pas contenus dans $\Wzer$ sont le localisateur fondamental trivial $\Wtr=\Fl\,\Cat$ et le localisateur fondamental grossier $\Wgr$~\cite[proposition 9.3.2]{CiAst}. On a des inclusions
\[
\Winf\subset\Wn{n}\subset\Wn{m}\subset\Wzer\subset\Wgr\subset\Wtr\smsp,\qquad m\leqslant n\smsp.
\]
\end{subparagr}

\begin{subparagr} \titsubparagr{Localisateur fondamental associé à une catégorie de modèles} \label{locfondcatmod}
Il existe beaucoup d'autres localisateurs fondamentaux. Par exemple, pour toute catégorie de modèles de Quillen $\C$, la classe $\Wmod{\C}$ des flèches $u:A\toto B$ de $\Cat$ induisant, pour tout objet $X$ de $\C$, un isomorphisme $\hocolim_AX\toto\hocolim_BX$ de la colimite homotopique du foncteur constant de valeur $X$ indexé par $A$, vers celle de celui indexé par $B$, est un localisateur fondamental. Cela est une conséquence immédiate des propriétés formelles des colimites homotopiques.
\end{subparagr}

\begin{subparagr} \titsubparagr{Localisateur fondamental associé à un dérivateur} \label{locfondderinf}
\`A tout dérivateur $\D$, on associe un localisateur fondamental $\Wder{\D}$~\cite{Der},~\cite{Mal1}. Cet exemple généralise le précédent. En effet, à toute catégorie de modèles de Quillen, on associe un dérivateur $\Dermod{\C}$~\cite{CiDer}, et on a (essentiellement par définition) $\Wder{\Dermod{\C}}=\Wmod{\C}$.
\end{subparagr}
\end{paragr}

\begin{rem}
On ne connaît pas d'exemple de localisateur fondamental faible qui ne soit pas un localisateur fondamental.
\end{rem}

\begin{paragr} \titparagr{Propriétés de stabilité des localisateurs fondamentaux} \label{stablocfond}
Si $\W$ est un localisateur fondamental faible (et \emph{a fortiori} s'il est un localisateur fondamental) une flèche \hbox{$u:A\toto B$} de $\Cat$ est une $\W$\nbd-équivalence si et seulement si le foncteur $\op{u}:\op{A}\toto\op{B}$, obtenu par passage aux catégories opposées, est une $\W$\nbd-équivalence~\cite[proposition 1.1.22]{Ast}. Autrement dit, on a $\op{\W}=\W$. Si $\W$ est un localisateur fondamental, alors il est stable par produits finis~\cite[proposition 2.1.3]{Ast}, par petites sommes~\cite[proposition 2.1.4]{Ast}, et par petites limites inductives filtrantes~\cite[proposition 2.4.12,~(\emph{b})]{Ast} et en particulier par rétractes.
\end{paragr}

\begin{paragr} \titparagr{Notions associées à un localisateur fondamental} Soit $\W$ un localisateur fondamental faible.

\begin{subparagr} \titsubparagr{Catégories asphériques} \label{catasph}
On dit qu'une petite catégorie $A$ est $\W$\nbd-\emph{asphérique}, ou plus simplement \emph{asphérique}, si le foncteur $A\toto e$ de $A$ vers la catégorie finale est une $\W$\nbd-équivalence. L'axiome LB affirme qu'une petite catégorie admettant un objet final est $\W$\nbd-asphérique, et il résulte de la stabilité de $\W$ par passage aux catégories opposées qu'une petite catégorie admettant un objet initial est $\W$\nbd-asphérique. La classe des petites catégories $\W$\nbd-asphériques est stable par passage à la catégorie opposée, par produits finis~\cite[corollaire 1.1.5]{Ast}, et par petites limites inductives filtrantes~\cite[proposition~2.4.12,~(\emph{a})]{Ast}.
\end{subparagr}

\begin{subparagr} \titsubparagr{Foncteurs asphériques, coasphériques} \label{asphcoasph}
Soit $u:A\toto B$ un morphisme de $\Cat$. On dit que le foncteur $u$ est $\W$\nbd-\emph{asphérique}, ou plus simplement \emph{asphérique}, si pour tout objet $b$ de $B$ le morphisme $\cm{A}{b}\toto\cm{B}{b}$, induit par $u$, est une $\W$\nbd-équivalence. Comme la catégorie $\cm{B}{b}$ admet un objet final, il résulte de LA et LB que cela revient à demander que la catégorie $\cm{A}{b}$ soit $\W$\nbd-asphérique. La condition LC${}_{\mathrm f}$ affirme qu'un foncteur $\W$\nbd-asphérique est une $\W$\nbd-équivalence. Un foncteur admettant un adjoint à droite est $\W$\nbd-asphérique~\cite[proposition 1.1.9]{Ast}.
\smallbreak

Dualement, on dit que le foncteur $u$ est $\W$\nbd-\emph{coasphérique} ou plus simplement \emph{coasphérique}, si pour tout objet $b$ de $B$ le morphisme $\mc{b}{A}\toto\mc{b}{B}$, induit par $u$, est une $\W$\nbd-équivalence, ou de façon équivalente (puisque la catégorie $\mc{b}{B}$ admet un objet initial) si la catégorie $\mc{b}{A}$ est $\W$\nbd-asphérique. Il résulte de l'isomorphisme canonique $\op{(\mc{b}{A})}\simeq\cm{\op{A}}{b}$ et de la stabilité de $\W$ par passage aux catégories opposées que le foncteur $u$ est $\W$\nbd-coasphérique si et seulement si le foncteur $\op{u}:\op{A}\toto\op{B}$ est $\W$\nbd-asphérique. En particulier, un foncteur $\W$\nbd-coasphérique est une $\W$\nbd-équivalence. Un foncteur admettant un adjoint à gauche est $\W$\nbd-coasphérique.
\smallbreak

Une petite catégorie $A$ est $\W$\nbd-asphérique si et seulement si le foncteur $A\toto e$ est $\W$\nbd-asphérique, ou de façon équivalente $\W$\nbd-coasphérique. Si $u:A\toto B$ est un foncteur $\W$\nbd-asphérique (resp. $\W$\nbd-coasphérique), alors pour tout objet $b$ de $B$, le foncteur $\cm{A}{b}\toto\cm{B}{b}$ (resp. $\mc{b}{A}\toto\mc{b}{B}$), induit par $u$, l'est aussi~\cite[lemme 1.1.7]{Ast}. Si
$$
\xymatrixcolsep{1pc}
\xymatrix{
A\ar[rr]^{u}\ar[dr]_{v}
&&B\ar[dl]^{w}
\\
&C
}$$
désigne un triangle commutatif de $\Cat$, et si $u$ est un foncteur $\W$\nbd-asphérique (resp. $\W$\nbd-coasphérique), alors pour que le morphisme $v$ soit $\W$\nbd-asphérique (resp. $\W$\nbd-coasphérique), il faut et il suffit que $w$ le soit~\cite[proposition 1.1.8]{Ast}. En particulier, la classe des foncteurs $\W$\nbd-asphériques (resp. $\W$\nbd-coasphériques) est stable par composition. Elle est également stable par produits finis~\cite[corollaire 1.1.6]{Ast}.
\end{subparagr}

\begin{subparagr} \titsubparagr{\'Equivalences faibles universelles} \label{equniv}
On dit qu'une flèche $u:A\toto B$ de $\Cat$ est une $\W$\nbd-\emph{équivalence universelle} si elle est une $\W$\nbd-équivalence et le reste après tout changement de base,
autrement dit, si pour tout carré cartésien
\[
\xymatrix{
&A'\ar[r]\ar[d]_{u'}
&A\ar[d]^{u}
\\
&B'\ar[r]
&B
&\hskip -30pt,
}
\]
$u'$ est une $\W$-équivalence. Une $\W$\nbd-équivalence universelle est en particulier un foncteur à la fois asphérique et coasphérique. Si $A$ est une catégorie asphérique alors le foncteur $A\toto e$ est une $\W$\nbd-équivalence universelle~\cite[proposition~1.1.4]{Ast}.
\end{subparagr}

\begin{subparagr} \titsubparagr{\'Equivalences faibles locales, colocales} \label{eqloccoloc}
\'Etant donné un triangle commutatif dans $\Cat$
$$
\xymatrixcolsep{1pc}
\xymatrix{
A\ar[rr]^{u}\ar[dr]_{v}
&&B\ar[dl]^{w}
\\
&C
&\hskip 10pt,\hskip -10pt
}$$
on dit que le foncteur $u$ est une $\W$\nbd-\emph{équivalence} (ou une \emph{équivalence faible}) \emph{localement} (resp. \emph{colocalement}) sur $C$ si pour tout objet $c$ de $C$, le foncteur $\cm{A}{c}\toto\cm{B}{c}$ (resp. $\mc{c}{A}\toto\mc{c}{B}$), induit par $u$, est une $\W$\nbd-équivalence. Le foncteur $u$ est une $\W$\nbd-équivalence colocalement sur $C$ si et seulement si le foncteur $\op{u}:\op{A}\toto\op{B}$ est une $\W$\nbd-équivalence localement sur $\op{C}$. Si le foncteur $u$ est $\W$\nbd-asphérique (resp. $\W$\nbd-coasphérique), alors il est \emph{a fortiori} une $\W$\nbd-équivalence localement (resp. colocalement) sur $C$~\cite[lemme~1.1.7]{Ast}. La condition LC affirme que si le foncteur $u$ est une $\W$\nbd-équivalence localement sur $C$, alors il est une $\W$\nbd-équivalence. Ainsi, si le localisateur fondamental faible $\W$ est un localisateur fondamental, un foncteur qui est une $\W$\nbd-équivalence localement ou colocalement sur $C$ est une $\W$\nbd-équivalence. Plus généralement, sous cette hypothèse, pour toute flèche $C\toto C'$ de $\Cat$, une $\W$\nbd-équivalence localement (resp. colocalement) sur $C$ l'est aussi sur $C'$~\cite[lemme~3.1.5]{Ast}. Un foncteur $u:A\toto B$ est $\W$\nbd-asphérique (resp. $\W$\nbd-coasphérique) si et seulement si il est une $\W$\nbd-équivalence localement (resp. colocalement) sur $B$. Pour que $u$ soit une $\W$\nbd-équivalence, il faut et il suffit qu'il soit une $\W$\nbd-équivalence localement (ou colocalement) sur la catégorie ponctuelle $e$.
\end{subparagr}

\begin{subparagr} \titsubparagr{Foncteurs propres, lisses} \label{defproprelisse}
Soit $u:A\toto B$ un morphisme de $\Cat$. Pour tout objet $b$ de $B$, on note $A_b$ la \emph{fibre} de $u$ en $b$, autrement dit, la sous-catégorie (non pleine) de $A$ dont les objets sont les objets $a$ de $A$ tels que $u(a)=b$, et dont les morphismes sont les flèches $f$ de $A$ telles que $u(f)=\id{b}$.  On dit que le foncteur $u$ est $\W$\nbd-\emph{lisse} (resp. $\W$\nbd-\emph{propre}), ou plus simplement \emph{lisse} (resp. \emph{propre}),  si pour tout objet $b$ de $B$, le morphisme canonique
\[
\begin{aligned}
&A_b\toto\mc{b}{A}\smsp,\quad a\mapstoto (a,\,\id{b}:b\textstyle\toto u(a))\smsp,\quad a\in\Ob A_b\smsp,\\
\hbox{(resp.}\quad
&A_b\toto\cm{A}{b}\smsp,\quad a\mapstoto (a,\,\id{b}:u(a)\textstyle\toto b)\smsp,\quad a\in\Ob A_b\smsp,\quad)
\end{aligned}
\]
est $\W$\nbd-asphérique (resp. $\W$\nbd-coasphérique)~\cite{Der}, \cite[section 3.2]{Ast}. Le foncteur \hbox{$u:A\toto B$} est propre si et seulement si le foncteur $\op{u}:\op{A}\toto\op{B}$ est lisse.
\smallbreak

Les foncteurs $u:A\toto B$ dans $\Cat$ faisant de $A$ une catégorie \emph{préfibrée} (resp.~\emph{pré\-co\-fibrée})~\cite[exposé VI]{SGA1} sont des exemples de foncteurs lisses (resp.~propres), car alors le morphisme canonique $A_b\toto\mc{b}{A}$ (resp.~$A_b\toto\cm{A}{b}$) admet un adjoint à droite (resp.~à gauche), et est donc en particulier $\W$\nbd-asphérique (resp. $\W$\nbd-coasphérique). On dira alors que $u$ est une \emph{préfibration} (resp.~une \emph{précofibration}).
\end{subparagr}
\end{paragr}

\begin{paragr} \titparagr{Notions associées à des localisateurs fondamentaux comparables} \label{changloc}
Soient $\W$ et $\W'$ deux localisateurs fondamentaux faibles tels que $\W\subset\W'$. Il est immédiat que toute catégorie $\W$\nbd-asphérique est $\W'$\nbd-asphérique, et que tout foncteur $\W$\nbd-asphérique (resp. $\W$\nbd-coasphérique) est $\W'$\nbd-asphérique (resp. $\W'$\nbd-coasphérique). En particulier, pour tout localisateur fondamental faible $\W$, une catégorie $\Winf$\nbd-asphérique est $\W$\nbd-asphérique, et un foncteur $\Winf$\nbd-asphérique (resp. $\Winf$\nbd-coasphérique) est $\W$\nbd-asphérique (resp. $\W$\nbd-coasphérique). De même, si $\W$ est un localisateur fondamental distinct de $\Wtr$ et de $\Wgr$, toute catégorie $\W$\nbd-asphérique est $\Wzer$\nbd-asphérique, et tout foncteur $\W$\nbd-asphérique (resp. $\W$\nbd-co\-asphé\-rique) est $\Wzer$\nbd-asphérique (resp. $\Wzer$\nbd-coasphérique).
\end{paragr}

\section{Théorie élémentaire des carrés exacts homotopiques}

\begin{paragr} \titparagr{\boldmath Carrés dans $\Cat$} \label{defcar}
On appellera \emph{carré} un \og$2$-diagramme\fg{} dans $\Cat$ de la forme
\[
\UseTwocells
\xymatrixcolsep{2.5pc}
\xymatrixrowsep{2.4pc}
\xymatrix{
&A'\ar[d]_{u'}\ar[r]^{v}
\drtwocell<\omit>{\alpha}
&A\ar[d]^{u}
\\
&B'\ar[r]_{w}
&B
&\hskip -35pt,
}
\] 
autrement dit, la donnée de quatre petites catégories $A,A',B,B'$, de quatre foncteurs $u:A\toto B$, $u':A'\toto B'$, $v:A'\toto A$, $w:B'\toto B$, et d'un morphisme de foncteurs $\alpha:uv\toto wu'$. On dira que ce carré est \emph{commutatif} si $\alpha=\id{uv}$, ce qui implique en particulier que $uv=wu'$, et on dira qu'il est \emph{cartésien}, s'il est commutatif, et si le morphisme induit de $A'$ vers le produit fibré $B'\times_BA$ est un isomorphisme de catégories. Enfin, on dira qu'il est un \emph{carré comma} si le morphisme canonique de $A'$ vers la catégorie comma $u\comma w$ est un isomorphisme. On rappelle que la catégorie $u\comma v$ a comme objets les triplets $(a,b',g)$, où $a$ est un objet de $A$, $b'$ un objet de $B'$, et $g:u(a)\toto w(b')$ une flèche de $B$, un morphisme d'un objet $(a_1,b'_1,g_1)$ vers un autre $(a_2,b'_2,g_2)$ étant un couple $(f,g')$, où $f:a_1\toto a_2$ est une flèche de $A$ et $g':b'_1\toto b'_2$ une flèche de $B'$, tel que le carré
\[
\xymatrix{
u(a_1)\ar[r]^{g_1}\ar[d]_{u(f)}
&w(b'_1)\ar[d]^{w(g')}
\\
u(a_2)\ar[r]_{g_2}
&w(b'_2)
}
\]
soit commutatif. Le morphisme canonique de $A'$ vers $u\comma w$ associe à un objet $a'$ de $A'$ le triplet $(v(a'),u'(a'),\alpha_{a'})$ et à une flèche $f'$ de $A'$ le couple $(v(f'),u'(f'))$.
\end{paragr}

\begin{paragr} \titparagr{Foncteurs induits par un carré} \label{fonctindcar}
Soit
\[
\UseTwocells
\xymatrixcolsep{2.5pc}
\xymatrixrowsep{2.4pc}
\xymatrix{
&A'\ar[d]_{u'}\ar[r]^{v}
\drtwocell<\omit>{\alpha}
&A\ar[d]^{u}
\\
&B'\ar[r]_{w}
&B
&
}
\] 
un carré de $\Cat$. Pour tout objet $b'$ de $B'$, le morphisme $v$ induit un foncteur $\cm{A'}{b'}\toto\cm{A}{w(b')}$, associant à un objet $(a',g':u'(a')\toto b')$ de $\cm{A'}{b'}$, l'objet
\[
\bigl(v(a'),\, uv(a')\Toto{2.99}{w(g')\alpha_{a'}}w(b')\bigr)
\]
de $\cm{A}{w(b')}$, et à une flèche $f'$, la flèche $v(f')$. Dualement, pour tout objet $a$ de $A$, le foncteur $u'$ induit un foncteur $\mc{a}{A'}\toto\mc{u(a)}{B'}$, associant à un objet $(a',f:a\toto v(a'))$ de $\mc{a}{A'}$, l'objet
\[
\bigl(u'(a'),\,u(a)\Toto{3}{\alpha_{a'}u(f)}wu'(a')\bigr)
\]
de $\mc{u(a)}{B'}$, et à une flèche $f'$, la flèche $u'(f')$. On remarquera que ces foncteurs induits dépendent, non seulement des foncteurs $v$ et $u'$ respectivement, \emph{mais aussi de la transformation naturelle} $\alpha$, même si on ne le précise pas quand aucune ambiguïté n'en résulte.
\end{paragr}

\noindent
{\it Dans la suite, on se fixe, une fois pour toutes, un localisateur
fondamental faible $\W$.}

\begin{prop} \label{carcarex}
Soit 
\[
\UseTwocells
\xymatrixcolsep{2.5pc}
\xymatrixrowsep{2.4pc}
\xymatrix{
A'\ar[d]_{u'}\ar[r]^{v}
\drtwocell<\omit>{\alpha}
&A\ar[d]^{u}
\\
B'\ar[r]_{w}
&B
}
\] 
un carré dans $\Cat$. Les conditions suivantes sont équivalentes:
\begin{GMitemize}
\item[\rm a)] pour tout objet $b'$ de $B'$, le foncteur $\cm{A'}{b'}\toto\cm{A}{w(b')}$, induit par $v$, est $\W$\nobreakdash-coasphérique;
\item[\rm b)] pour tout objet $a$ de $A$, le foncteur $\mc{a}{A'}\toto\mc{u(a)}{B'}$, induit par $u'$, est $\W$\nobreakdash-asphérique;
\item[\rm c)] pour tout objet $a$ de $A$, tout objet $b'$ de $B'$, et toute flèche $g:u(a)\toto w(b')$ de $A$, la catégorie $A'_{(a,b',g)}$ dont les objets sont les triplets $(a',f,g')$, où $a'$ est un objet de $A'$, $f:a\toto v(a')$ une flèche de $A$ et $g':u'(a')\toto b'$ une flèche de $B'$, tels que $w(g')\alpha_{a'}u(f)=g$, et les morphismes
\[
(a'_1,f_1,g'_1)\Toto{1.99}{f'}(a'_2,f_2,g'_2)
\]
les flèches $f':a'_1\toto a'_2$ de $A'$ rendant commutatifs les deux triangles suivants
\[
\xymatrixrowsep{.5pc}
\xymatrix{
&&v(a'_1)\ar[dd]^{v(f')}
&&u'(a'_1)\ar[dd]_{u'(f')}\ar[rd]^-{g'_1}
\\
&a\ar[ru]^-{f_1}\ar[rd]_-{f_2}
&&&&b'
\\
&&v(a'_2)
&&u'(a'_2)\ar[ru]_-{g'_2}
&&\hskip -30pt,
}
\]
est $\W$-asphérique.
\end{GMitemize}
\end{prop}

\begin{proof}
Une vérification immédiate montre que pour tout objet $b'$ de $B'$, et tout objet $(a,g:u(a)\toto w(b'))$ de $\cm{A}{w(b')}$, la catégorie $\mc{(a,g)}{(\cm{A'}{b'})}$ est isomorphe à la catégorie $A_{(a,b',g)}$, ce qui prouve l'équivalence des conditions~(\emph{a}) et~(\emph{c}). Dualement, pour tout objet $a$ de $A$, et tout objet $(b',g:u(a)\toto w(b'))$ de $\mc{u(a)}{B'}$ la catégorie $\cm{(\mc{a}{A'})}{(b',g)}$ est isomorphe à la catégorie $A_{(a,b',g)}$, ce qui achève la démonstration.
\end{proof}

\begin{paragr} \titparagr{\boldmath Carrés $\W$-exacts} \label{defcarex}
On dit qu'un carré est $\W$-\emph{exact} s'il satisfait aux conditions équivalentes de la proposition ci-dessus. 
Les carrés exacts étudiés par Guitart~\cite{Guit1,Guit2,GuitSplit,Guit3,Guit4,Guit5} sont exactement les carrés $\Wzer$\nbd-exacts~\cite[théorème~1.3]{Guit1}, qu'on appellera aussi \emph{carrés exacts au sens de Guitart}~\cite[section 4]{BKGM}, ou plus simplement \emph{carrés exacts de Guitart}.
\end{paragr}

\begin{paragr} \titparagr{\boldmath Exemples triviaux de carrés $\W$-exacts} \label{extrcarex}
Un carré de la forme
\[
\UseTwocells
\xymatrixcolsep{2.5pc}
\xymatrixrowsep{2.4pc}
\xymatrix{
A\ar[d]_{u}\ar[r]^{}
\drtwocell<\omit>{}
&e\ar[d]^{}\ar@<-2ex>@{}[d]_(.55){=}\ar@{}[d]^{\hbox{\kern 30pt(resp.}}
\\
B\ar[r]_{}
&e
}
\quad
\xymatrix{
A\ar[d]_{}\ar[r]^{u}
\drtwocell<\omit>{}
&B\ar[d]^{\hbox{\kern 20pt),}}\ar@<-2ex>@{}[d]_(.55){=}
\\
e\ar[r]_{}
&e
}
\]
où $e$ désigne la catégorie ponctuelle, est $\W$\nbd-exact si et seulement si le foncteur $u$ est $\W$\nbd-asphérique (resp. $\W$\nbd-coasphérique). En particulier, un carré de la forme
\[
\UseTwocells
\xymatrixcolsep{2.5pc}
\xymatrixrowsep{2.4pc}
\xymatrix{
A\ar[d]_{}\ar[r]^{}
\drtwocell<\omit>{}
&e\ar[d]^{}\ar@<-2ex>@{}[d]_(.55){=}
\\
e\ar[r]_{}
&e
}
\]
est $\W$-exact si et seulement si la catégorie $A$ est $\W$\nbd-asphérique.
\end{paragr}

\begin{prop} \label{carexop}
Pour qu'un carré
\[
\UseTwocells
\xymatrixcolsep{2.5pc}
\xymatrixrowsep{2.4pc}
\xymatrix{
A'\ar[d]_{u'}\ar[r]^{v}
\drtwocell<\omit>{\alpha}
&A\ar[d]^{u}
\\
B'\ar[r]_{w}
&B
}
\] 
soit $\W$\nbd-exact, il faut et il suffit que le carré
\[
\UseTwocells
\xymatrixcolsep{2.5pc}
\xymatrixrowsep{2.4pc}
\xymatrix{
&\op{A'}\ar[r]^{\op{u'}}\ar[d]_{\op{v}}
\drtwocell<\omit>{\op{\alpha}}
&\op{B'}\ar[d]^{\op{w}}
\\
&\op{A}\ar[r]_{\op{u}}
&\op{B}
&\hskip -35pt,
}
\] 
obtenu du précédent par passage aux catégories opposées, soit $\W$\nbd-exact.
\end{prop}

\begin{proof}
La proposition résulte aussitôt de l'équivalence des conditions~(\emph{a}) et~(\emph{b}) de la proposition~\ref{carcarex}, et du fait qu'un foncteur est $\W$\nbd-coasphérique si et seulement si le foncteur obtenu par passage aux catégories opposées est $\W$\nbd-asphérique (\cf~\ref{asphcoasph}).
\end{proof}

\begin{prop} \label{compcarhex}
La classe des carrés $\W$\nbd-exacts est stable par composition horizontale et verticale.
\end{prop}

\begin{proof}
Montrons par exemple la stabilité par composition horizontale. Considérons donc deux carrés composables horizontalement et leur composé
\[
\xymatrixrowsep{1.1pc}
\xymatrixcolsep{.9pc}
\UseAllTwocells
\xymatrix{
A''\ddrrcompositemap<\omit>{\alpha'}
  \ar[rr]^{v'}
  \ar[dd]_{u''}
&&A'\ddrrcompositemap<\omit>{\alpha}
  \ar[rr]^{v}
  \ar[dd]^{u'}
&&A\ar[dd]^u
&&&A''\ddrrcompositemap<\omit>{\alpha''}
  \ar[rr]^{v''=vv'}
  \ar[dd]_{u''}
&&A\ar[dd]^u
\\
\\
B''\ar[rr]_{w'}
&&B'\ar[rr]_w
&&B
&&&B''\ar[rr]_{w''=ww'\,\,}
&&B
&,
\\
&\mathcal D'
&&\mathcal D
&&&&&\kern -10pt \mathcal D\circh\mathcal D'\kern -10pt
}
\]
avec $\alpha'' =(w\star\alpha')(\alpha\star v')$, 
supposons que les carrés $\mathcal D$ et $\mathcal D'$ soient $\W$\nbd-exacts, et montrons qu'il en est de même du carré $\mathcal D\circh\mathcal D'$. Soit $b''$ un objet de $B''$. On vérifie aussitôt que le foncteur $\cm{A''}{b''}\toto\cm{A}{w''(b'')}$, induit par la flèche $v''$ du carré $\mathcal D\circh\mathcal D'$, est le composé 
\[
\cm{A''}{b''}\toto\cm{A'}{w'(b'')}\toto\cm{A}{w(w'(b''))}=\cm{A}{w''(b'')}
\]
des foncteurs induits par les flèches $v'$ et $v$ des carrés $\mathcal D$ et $\mathcal D'$ respectivement. L'assertion résulte donc du critère (\emph{a}) de la proposition~\ref{carcarex}, et de la stabilité des foncteurs $\W$\nbd-coasphériques par composition (\cf~\ref{asphcoasph}). La stabilité des carrés $\W$\nbd-exacts par composition verticale résulte de ce qui précède et de la proposition~\ref{carexop}.
\end{proof}

\begin{prop} \label{desccarhexh}
Soient $J$ un ensemble, et 
$$
\xymatrixrowsep{.7pc}
\xymatrixcolsep{1.3pc}
\UseAllTwocells
\xymatrix{
&A'\ddrrcompositemap<\omit>{\alpha}
  \ar[rr]^{v}
  \ar[dd]_{u'}
&&A\ar[dd]^u
&&&A_j\ddrrcompositemap<\omit>{\ \alpha_j}
  \ar[rr]^{v_j}
  \ar[dd]_{u_j}
&&A'\ar[dd]^{u'}
\\
{\mathcal D}\ =
&&&&\hbox{et}
&{\mathcal D_j}\ =
&&&&,\ j\in J\ ,
\\
&B'\ar[rr]_w
&&B
&&&B_j\ar[rr]_{w_j}
&&B'
}
$$
des carrés dans $\Cat$. On suppose que pour tout $j\in J$ le carré $\mathcal D_j$ est $\W$\nbd-exact, et que 
$\Ob\,B'=\bigcup\limits_{j\in J}w_j(\Ob\,B_j)$.
Alors les conditions suivantes sont équivalentes:
\begin{GMitemize}
\item[\rm a)] le carré $\mathcal D$ est $\W$\nbd-exact;
\item[\rm b)] pour tout $j$, $j\in J$, le carré composé $\mathcal D\circh\mathcal D_j$ est $\W$\nbd-exact.
\end{GMitemize}
\end{prop}

\begin{proof}
L'implication (\emph{a}) $\Rightarrow$ (\emph{b}) résulte de la proposition \ref{compcarhex}. Montrons la réciproque. Soit $b'$ un objet de $B'$. Par hypothèse, il existe $j\in J$ et un objet $b_j$ de $B_j$ tels que $b'=w_j(b_j)$. Comme les carrés $\mathcal D_j$ et $\mathcal D\circh\mathcal D_j$ sont $\W$\nbd-exacts, il résulte du critère (\emph{a}) de la proposition~\ref{carcarex} que le morphisme $\cm{B_j}{b_j}\toto\cm{B'}{b'}$, ainsi que le morphisme composé
\[
\cm{B_j}{b_j}\toto\cm{B'}{b'}\toto\cm{B}{w(b')}
\]
sont $\W$\nbd-coasphériques. Il en est donc de même du morphisme $\cm{B'}{b'}\toto\cm{B}{w(b')}$ (\cf~\ref{asphcoasph}). Une nouvelle application du critère (\emph{a}) de la proposition~\ref{carcarex} montre alors que le carré $\mathcal D$ est $\W$\nbd-exact.
\end{proof}

\begin{rem} \label{desccarhexv}
On laisse le soin au lecteur d'énoncer et prouver l'assertion duale de la proposition précédente, concernant la composition verticale.
\end{rem}

\begin{prop} \label{carexcomma}
Tout carré comma est $\W$-exact.
\end{prop}

\begin{proof}
Soient $A,B,B'$ trois catégories, $u:A\toto B$, $w:B'\toto B$ deux foncteurs, et formons le carré comma
\[
\UseTwocells
\xymatrixcolsep{2.5pc}
\xymatrixrowsep{2.4pc}
\xymatrix{
&A'\ar[d]_{u'}\ar[r]^{v}
\drtwocell<\omit>{\alpha}
&A\ar[d]^{u}
\\
&B'\ar[r]_{w}
&B
&\hskip -30pt,
}
\]
la catégorie $A'$ ayant comme objets les triplets $(a,\,b',\,g:u(a)\toto w(b'))$, $a\in \Ob\, A$, \hbox{$b'\in \Ob\, B'$}, $g\in\Fl\, B$, une flèche $(a^{}_1,b'_1,g^{}_1)\toto(a^{}_2,b'_2,g^{}_2)$ étant un couple $(f,g')$, \hbox{$f:a^{}_1\toto a^{}_2\in\Fl\, A$}, $g':b'_1\toto b'_2\in\Fl\, B'$, tel que $g^{}_2u(f)=w(g')g^{}_1$, et les foncteurs $v$, $u'$ et la transformation naturelle $\alpha$ étant définis par les formules
\[
\begin{aligned}
&v(a,b',g)=a\smsp,\quad\hskip 5pt(a,b',g)\in\Ob\, A'\smsp,\qquad v(f,g')=f\smsp,\quad\hskip 5.5pt(f,g')\in\Fl\, A'\smsp,\\
\noalign{\vskip 3pt plus 1pt minus 1pt}
&u'(a,b',g)=b'\smsp,\quad(a,b',g)\in\Ob\, A'\smsp,\qquad u'(f,g')=g'\smsp,\quad(f,g')\in\Fl\, A'\smsp,\\
\noalign{\vskip 3pt plus 1pt minus 1pt}
&\alpha_{(a,b',g)}=g\smsp,\quad\hskip 12pt(a,b',g)\in\Ob\, A'\smsp.
\end{aligned}
\]
Pour montrer que ce carré est $\W$\nbd-exact, on va utiliser le critère (\emph{c}) de la proposition~\ref{carcarex}. Dans les notations de ce critère, il s'agit de prouver que pour tout objet $a^{}_0$ de $A$, tout objet $b'_0$ de $B'$, et toute flèche $g^{}_0:u(a^{}_0)\toto w(b'_0)$ de $B$, la catégorie $C=A'_{(a^{}_0,b'_0,g^{}_0)}$ est $\W$\nbd-asphérique. Les objets de $C$ 
sont les quintuplets
\[
(a,\ b',\ u(a)\Toto{1.3}{g}w(b'),\ a^{}_0\Toto{1.3}{f}a,\ b'\Toto{1.3}{g'}b'_0)\smsp, 
\]
avec $a\in\Ob\,A$, $b'\in\Ob\,B'$, $g\in\Fl\,B$, $f\in\Fl\,A$, $g'\in\Fl\,B'$, tels que 
\[
w(g')\,g\,u(f)=g^{}_0\smsp,
\]
un morphisme
\[
(a^{}_1,\,b'_1,\,g^{}_1,\,f^{}_1,\,g'_1)\toto(a^{}_2,\,b'_2,\,g^{}_2,\,f^{}_2,\,g'_2)
\]
étant un couple $(f,g')$, $f:a^{}_1\toto a^{}_2\in\Fl\,A$, $g':b'_1\toto b'_2$, tel que les diagrammes
\[
\xymatrixrowsep{.5pc}
\xymatrix{
&a^{}_1\ar[dd]^{f}
&&u(a^{}_1)\ar[r]^{g^{}_1}\ar[dd]_{u(f)}
&w(b'_1)\ar[dd]^{w(g')}
&&b'_1\ar[dd]_{g'}\ar[rd]^{g'_1}
\\
a^{}_0\ar[ru]^{f^{}_1}\ar[rd]_{f^{}_2}
&&&&&&&b'_0
\\
&a^{}_2
&&u(a^{}_2)\ar[r]_{g^{}_2}
&w(b'_2)
&&b'_2\ar[ru]_{g'_2}
}
\]
soient commutatifs. On va construire un décalage sur $C$
\[
\id{C}\Toto{1.3}{\alpha\,}D\Otot{1.3}{\,\,\beta}K
\]
(diagramme d'endofoncteurs de $C$ et transformations naturelles, avec $K$ endofoncteur constant~\cite[3.1]{CiGM}), ce qui impliquera que la catégorie $C$ est $\Winf$\nbd-asphérique~\cite[proposition~3.6]{CiGM} et en particulier $\W$\nbd-asphérique (\cf~\ref{changloc}).
\smallbreak

\noindent\textbf{\boldmath Définition du foncteur $D$.} 
L'endofoncteur $D$ est défini par les formules:
\[
\begin{aligned}
&D(a,b,g,f,g')=(a,b'_0,w(g')g,f,\id{b'_0})\smsp,\quad(a,b,g,f,g')\in\Ob\,C\smsp,\\
\noalign{\vskip 3pt plus 1pt minus 1pt}
&D(f,g')=(f,\id{b'_0})\smsp,\hskip 93.7pt(f,g')\in\Fl\,C\smsp.
\end{aligned}
\]
\smallbreak

\noindent\textbf{\boldmath Définition du morphisme de foncteurs $\alpha$.} La transformation naturelle $\alpha:\id{C}\toto D$ est définie par l'égalité:
\[
\alpha_{(a,b,g,f,g')}=(\id{a},g')\smsp,\qquad(a,b,g,f,g')\in\Ob\,C\smsp.
\]
\smallbreak

\noindent\textbf{\boldmath Définition du foncteur $P$.} Le foncteur $P$ est l'endofoncteur constant de $C$ défini par l'objet 
\vskip -15pt
\[
(a^{}_0,\,b'_0,\,u(a^{}_0)\Toto{1.3}{g^{}_0}w(b'_0),\,\id{a^{}_0},\,\id{b'_0})\smsp.
\]
\smallbreak

\noindent\textbf{\boldmath Définition du morphisme de foncteurs $\beta$.} La transformation naturelle $\beta:P\toto D$ est définie par l'égalité
\[
\beta_{(a,b,g,f,g')}=(f,\id{b'_0})\smsp,\qquad(a,b,g,f,g')\in\Ob\,C\smsp.
\]
On laisse le soin au lecteur de vérifier que ces formules définissent bien des endofoncteurs de $C$ et des transformations naturelles.
\end{proof}

\begin{rem}
Guitart montre que tout carré \emph{cocomma} est $\Wzer$\nbd-exact~\cite{Guit1}. Ce résultat ne se généralise \emph{pas} aux carrés $\W$\nbd-exacts, pour un localisateur fondamental distinct de $\Wzer$, $\Wgr$ et $\Wtr$. En effet, alors $\W\subsetneqq\Wzer$ (\cf~\ref{defWgr}), et la stabilité par sommes de $\W$ (\cf~\ref{stablocfond}) implique l'existence d'une petite catégorie $A$, $0$\nbd-connexe mais non $\W$\nbd-asphérique. Or, pour toute catégorie $0$\nbd-connexe $A$, le carré 
\[
\UseAllTwocells
\xymatrixcolsep{.0pc}
\xymatrixrowsep{.3pc}
\xymatrix{
A\ar[dddd]_{}\ar[rrrrrr]^{}
&&&&&&e\ar[dddd]^{0}
\\&&&
\ddrrcompositemap<\omit>{}
&&&
\\&&&&&&
\\&&&&&&
\\
e\ar[rrrrr]_-{1}
&&&&&\{0\ar[rr]
&{\phantom{\}}}
&1\}
}
\]
est un carré cocomma. En revanche, en vertu du critère (\emph{b}) de la proposition~\ref{carcarex}, ce carré n'est \emph{pas} $\W$\nbd-exact si la catégorie $A$ n'est pas $\W$\nbd-asphérique.
\end{rem}

\begin{prop} \label{proprelisseex}
Soit
\[
\mathcal{D}\ =\quad
\raise 23pt
\vbox{
\UseTwocells
\xymatrixcolsep{2.5pc}
\xymatrixrowsep{2.4pc}
\xymatrix{
A'\ar[d]_{u'}\ar[r]^{v}
\drtwocell<\omit>{}
&A\ar[d]^{u}\ar@<-2ex>@{}[d]_(.55){=}
\\
B'\ar[r]_{w}
&B
}
}\qquad
\]
un carré cartésien de $\Cat$. Si $u$ est $\W$-propre ou si $w$ est $\W$\nbd-lisse, alors le carré $\mathcal D$ est $\W$\nbd-exact.
\end{prop}

\begin{proof}
Supposons que $u$ soit $\W$-propre, et soit $a$ un objet de $A$. La flèche $\mc{a}{A}\toto\mc{u(a)}{B}$ est alors une $\W$-équivalence universelle~\cite[proposition 3.2.8]{Ast}. Comme le carré
\[
\xymatrixcolsep{1.8pc}
\xymatrixrowsep{2.pc}
\xymatrix{
&\mc{a}{A'}\ar[r]\ar[d]
&\mc{a}{A}\ar[d]
\\
&\mc{u(a)}{B'}\ar[r]
&\mc{u(a)}{B}
&\hskip -30pt,
}
\]
induit par $\mathcal D$, est cartésien~\cite[lemme 3.2.11]{Ast}, il en est de même de la flèche $\mc{a}{A'}\toto\mc{u(a)}{B'}$, qui est donc en particulier $\W$\nbd-asphérique. Le critère (\emph{b}) de la proposition~\ref{carcarex} implique alors que le carré $\mathcal D$ est $\W$\nbd-exact. Le cas où $w$ est lisse se déduit de ce qui précède et de la proposition~\ref{carexop} (\cf~\ref{defproprelisse}).
\end{proof}

\begin{prop} \label{carpropreex}
Soit $u:A\toto B$ une flèche de $\Cat$. Les conditions suivantes sont équivalentes:
\begin{GMitemize}
\item[\rm a)] $u$ est $\W$-propre;
\item[\rm b)] tout carré cartésien de la forme 
\[
\UseTwocells
\xymatrixcolsep{2.5pc}
\xymatrixrowsep{2.4pc}
\xymatrix{
A'\ar[d]_{u'}\ar[r]^{v}
\drtwocell<\omit>{}
&A\ar[d]^{u}\ar@<-2ex>@{}[d]_(.55){=}
\\
B'\ar[r]_{w}
&B
}
\]
est $\W$-exact.
\end{GMitemize}
\end{prop}

\begin{proof}
L'implication (\emph{a}) $\Rightarrow$ (\emph{b}) résulte de la proposition~\ref{proprelisseex}. Pour prouver l'implication (\emph{b}) $\Rightarrow$ (\emph{a}), soit $b$ un objet de $B$, et considérons le carré cartésien 
\[
\UseTwocells
\xymatrixcolsep{2.5pc}
\xymatrixrowsep{2.4pc}
\xymatrix{
&A_b\ar[r]\ar[d]
\drtwocell<\omit>{}
&A\ar[d]^{u}\ar@<-2ex>@{}[d]_(.55){=}
\\
&e\ar[r]_{b}
&B
&\hskip -30pt,
}
\]
où $b:e\toto B$ désigne le foncteur de la catégorie ponctuelle vers $B$, défini par l'objet $b$ de $B$, et $A_b$ la fibre de $u$ en $b$. En vertu de la condition (\emph{b}), ce carré est $\W$\nbd-exact, et il résulte alors du critère (\emph{a}) de la proposition~\ref{carcarex} que la flèche $A_b\toto\cm{A}{b}$ est $\W$\nbd-coasphérique, ce qui prouve que $u$ est $\W$\nbd-propre.
\end{proof}

\begin{rem} \label{carlisseex}
On laisse le soin au lecteur d'énoncer et prouver l'assertion duale de la proposition précédente, caractérisant les foncteurs lisses.
\end{rem}

\begin{prop} \label{changloccarex}
Si $\W$ et $\W'$ sont deux localisateurs fondamentaux faibles tels que $\W\subset\W'$, alors tout carré $\W$\nbd-exact est $\W'$\nbd-exact.
\end{prop}

\begin{proof}
La proposition est conséquence immédiate des critères de la proposition~\ref{carcarex}, et des considérations de~\ref{changloc}.
\end{proof}

\begin{paragr} \titparagr{Carrés de Beck-Chevalley} \label{BeckChev}
On dit qu'un carré
\[
\mathcal{D}\ =\quad
\raise 23pt
\vbox{
\UseTwocells
\xymatrixcolsep{2.5pc}
\xymatrixrowsep{2.4pc}
\xymatrix{
A'\ar[d]_{u'}\ar[r]^{v}
\drtwocell<\omit>{\alpha}
&A\ar[d]^{u}
\\
B'\ar[r]_{w}
&B
}
}\qquad
\]
de $\Cat$ est de \emph{Beck-Chevalley à gauche} si $u$ et $u'$ admettent des adjoints \emph{à droite} $r$ et $r'$ respectivement, et si le morphisme canonique \og de changement de base\fg{} $c:vr'\toto rw$, défini par la formule
\[
\raise -20pt
\hbox{
$c=(rw\star\varepsilon')(r\star\alpha\star r')(\eta\star vr')$
}
\hskip 40pt
\xymatrix{
vr'\ar[r]^{c}\ar[d]_{\eta\star vr'}
&rw
\\
ruvr'\ar[r]_{r\star\alpha\star r'}
&rwu'r'\ar[u]_{rw\star\varepsilon'}
&\hskip -30pt,
}
\]
où
\[
\varepsilon:ur\toto\id{B}\smsp,\quad\eta:\id{A}\toto ru\smsp,\qquad\quad
\varepsilon':u'r'\toto\id{B'}\smsp,\quad\eta':\id{A'}\toto r'u'
\]
désignent les morphismes d'adjonction, est un isomorphisme. On vérifie facilement que cette propriété est indépendante du choix des foncteurs adjoints $r$ et $r'$, ainsi que du choix des morphismes d'adjonction. Dualement, on dit que le carré $\mathcal D$ est de \emph{Beck-Chevalley à droite} si le carré obtenu par passage aux catégories opposées est de Beck-Chevalley à gauche. Cela revient à demander que les foncteurs $v$ et $w$ admettent des adjoints \emph{à gauche} $v'$ et $w'$ respectivement, et que le morphisme canonique \hbox{$c':w'u\toto u'v'$}, défini de façon analogue, est un isomorphisme. Si les foncteurs $u$ et $u'$ admettent des adjoints à droite \emph{et} les foncteurs $v$ et $w$ des adjoints à gauche, alors le carré $\mathcal D$ est de Beck-Chevalley à droite si et seulement si il est de Beck-Chevalley à gauche, puisque alors $c$ et $c'$ sont transposés l'un de l'autre.
\end{paragr}

\begin{prop} \label{carcarBeckChev}
Soit 
\[
\mathcal{D}\ =\quad
\raise 23pt
\vbox{
\UseTwocells
\xymatrixcolsep{2.5pc}
\xymatrixrowsep{2.4pc}
\xymatrix{
A'\ar[d]_{u'}\ar[r]^{v}
\drtwocell<\omit>{\alpha}
&A\ar[d]^{u}
\\
B'\ar[r]_{w}
&B
}
}\qquad
\] 
un carré dans $\Cat$ tel que $u$ et $u'$ admettent des adjoints à droite. Alors les conditions suivantes sont équivalentes:
\begin{GMitemize}
\item[\rm a)] le carré $\mathcal D$ est de Beck-Chevalley à gauche;
\item[\rm b)] le carré $\mathcal D$ est $\Winf$\nbd-exact;
\item[\rm c)] le carré $\mathcal D$ est $\Wzer$\nbd-exact;
\item[\rm d)] pour tout objet $a$ de $A$, le foncteur $\mc{a}{A'}\toto\mc{u(a)}{B'}$, induit par $u'$, admet un adjoint à droite;
\item[\rm e)] pour tout objet $a$ de $A$, tout objet $b'$ de $B'$, et toute flèche $g:u(a)\toto w(b')$ de $A$, la catégorie $A'_{(a,b',g)}$ \emph{(du critère (\emph{c}) de la proposition~\ref{carcarex})} admet un objet final.
\end{GMitemize}
\end{prop}

\begin{proof}
Soit $a$ un objet de $A$. Dire que le foncteur $\mc{a}{A'}\toto\mc{u(a)}{B'}$ admet un adjoint à droite équivaut à dire que pour tout objet $(b',\,g:u(a)\toto w(b'))$ de $\mc{u(a)}{B'}$, la catégorie $\cm{(\mc{a}{A'})}{(b',g)}$ admet un objet final. Comme cette dernière est isomorphe à la catégorie $A'_{(a,b',g)}$, cela prouve l'équivalence des conditions~(\emph{d}) et~(\emph{e}). L'implication (\emph{d}) $\Rightarrow$ (\emph{b}) résulte du critère (\emph{b}) de la proposition~\ref{carcarex}, et du fait qu'un foncteur admettant un adjoint à droite est $\Winf$\nbd-asphérique (\cf~\ref{asphcoasph}). L'implication (\emph{b}) $\Rightarrow$ (\emph{c}) est un cas particulier de la proposition~\ref{changloccarex}.
\smallbreak

Il reste à prouver les implications (\emph{a}) $\Rightarrow$ (\emph{d}) et (\emph{c}) $\Rightarrow$ (\emph{a}). Choisissons $r$ et $r'$ des adjoints à droite de $u$ et $u'$ respectivement, et des morphismes d'adjonction
\[
\varepsilon:ur\toto\id{B}\smsp,\quad\eta:\id{A}\toto ru\smsp,\qquad\quad
\varepsilon':u'r'\toto\id{B'}\smsp,\quad\eta':\id{A'}\toto r'u'\smsp.
\]
La condition (\emph{a}) signifie que le morphisme de foncteurs
\[
c=(rw\star\varepsilon')(r\star\alpha\star r')(\eta\star vr'):vr'\toto rw
\]
est un isomorphisme. Pour tout objet $a$ de $A$, on vérifie alors facilement que le foncteur
\[
\mc{u(a)}{B'}\toto\mc{a}{A'}\smsp,\qquad \bigl(b',u(a)\Toto{1.3}{g}w(b')\bigr)\,\mapstoto\,\bigl(r'(b'),\,a\Toto{3.5}{c_{b'}^{-1}r(g)\eta^{}_a\,}vr'(b')\bigr)\smsp,
\]
est un adjoint à droite du foncteur $\mc{a}{A'}\toto\mc{u(a)}{B'}$, induit par $u'$, ce qui prouve l'implication (\emph{a}) $\Rightarrow$ (\emph{d}).
\smallbreak

L'implication (\emph{c}) $\Rightarrow$ (\emph{a}) résultera de la théorie générale des dérivateurs (\cf~remarque~\ref{excarBeckChev}). Voici néanmoins une preuve élémentaire. Soit $b'$ un objet de $B'$. On observe que les isomorphismes $\cm{A'}{b'}\simeq\cm{A'}{r'(b')}$ et $\cm{A}{w(b')}\simeq\cm{A}{rw(b')}$, déduits des adjonctions, identifient le foncteur $\cm{A'}{b'}\toto\cm{A}{w(b')}$, induit par $v$, au foncteur
\[
\cm{A'}{r'(b')}\toto\cm{A}{rw(b')}\smsp,\qquad\bigl(a',\,a'\Toto{1.3}{f'}r'(b')\bigr)\,\mapstoto\,\bigl(v(a'),\,v(a')\Toto{2.5}{c^{}_{b'}v(f')}rw(b')\bigr)\smsp.
\]
En vertu du critère (\emph{a}) de la proposition~\ref{carcarex}, la condition (\emph{c}) signifie que pour tout objet $(a,\,f_0:a\toto rw(b'))$ de $\cm{A}{rw(b')}$, la catégorie $\mc{(a,f_0)}{\bigl(\cm{A'}{r'(b')}\bigr)}$ est $0$\nbd-connexe. Décrivons cette dernière. Ses objets sont les triplets 
\[
\bigl(a',\,a'\Toto{1.3}{f'}r(b'),\,a\Toto{1.3}{f}v(a')\bigr)\smsp,\qquad a'\in\Ob\,A'\smsp,\quad f'\in\Fl\,A'\smsp,\quad f\in\Fl\,A\smsp,
\]
tels que $f_0=c^{}_{b'}v(f')f$. Un morphisme de $(a'_1,f'_1,f^{}_1)$ vers $(a'_2,f'_2,f^{}_2)$ est une flèche $f':a'_1\toto a'_2$ de $A'$ rendant commutatifs les triangles suivants
\[
\xymatrixrowsep{.5pc}
\xymatrix{
&&v(a'_1)\ar[dd]^{v(f')}
&&a'_1\ar[dd]_{f'}\ar[rd]^-{f'_1}
\\
&a\ar[ru]^-{f_1}\ar[rd]_-{f_2}
&&&&r'(b')
\\
&&v(a'_2)
&&a'_2\ar[ru]_-{f'_2}
&&\hskip -30pt.
}
\]
En particulier cela implique que 
\begin{equation} \label{eqconn}
v(f'_2)f^{}_2=v(f'_1)f^{}_1\smsp,
\end{equation}
et comme la catégorie $\mc{(a,f_0)}{\bigl(\cm{A'}{r'(b')}\bigr)}$ est connexe, on a cette égalité pour \emph{tout} couple d'objets $(a'_1,f'_1,f^{}_1)$ et $(a'_2,f'_2,f^{}_2)$. Considérons l'objet final $(rw(b'),\id{rw(b')})$ de $\cm{A}{rw(b')}$. Comme la catégorie $\mc{(rw(b'),\id{rw(b')})}{\bigl(\cm{A'}{r'(b')}\bigr)}$ est $0$\nbd-connexe, elle est en particulier non vide. Soit donc $(a',f',f)$ un objet de cette catégorie, de sorte que $\id{rw(b')}=c^{}_{b'}v(f')f$. On va montrer qu'on a aussi $v(f')fc^{}_{b'}=\id{vr'(b')}$, ce qui prouvera que $c^{}_{b'}$ est un isomorphisme. Pour cela, considérons l'objet 
\[
\bigl(vr'(b'),\,vr'(b')\Toto{1.5}{c^{}_{b'}}rw(b')\bigr)
\]
de $\cm{A}{rw(b')}$ et les objets
\[
\begin{aligned}
&\bigl(r'(b'),\,r'(b')\Toto{3}{\id{r'(b')}}r'(b'),\,vr'(b')\Toto{3}{\id{vr'(b')}}vr'(b')\bigr)\smsp,\\
&\bigl(r'(b'),\,r'(b')\Toto{3}{\id{r'(b')}}r'(b'),\,vr'(b')\Toto{3}{v(f')fc^{}_{b'}}vr'(b')\bigr)\smsp
\end{aligned}
\]
de la catégorie $\mc{(vr'(b'),c^{}_{b'})}{\bigl(\cm{A'}{r'(b')}\bigr)}$.
L'égalité \ref{eqconn} implique alors l'assertion, ce qui achève la démonstration.
\end{proof}

\noindent
\emph{Dans la suite, on suppose que le localisateur fondamental faible $\W$ est un localisateur fondamental.}

\begin{prop} \label{carcarexf}
Soit 
\[
\mathcal{D}\ =\quad
\raise 23pt
\vbox{
\UseTwocells
\xymatrixcolsep{2.5pc}
\xymatrixrowsep{2.4pc}
\xymatrix{
A'\ar[d]_{u'}\ar[r]^{v}
\drtwocell<\omit>{\alpha}
&A\ar[d]^{u}
\\
B'\ar[r]_{w}
&B
&\hskip -30pt
}
}
\] 
un carré dans $\Cat$. Les conditions suivantes sont équivalentes:
\begin{GMitemize}
\item[\rm a)] pour tout objet $b'$ de $B'$, le foncteur $\cm{A'}{b'}\toto\cm{A}{w(b')}$, induit par $v$, est une $\W$\nobreakdash-équivalence colocalement sur $A$;
\item[\rm b)] pour tout objet $a$ de $A$, le foncteur $\mc{a}{A'}\toto\mc{u(a)}{B'}$, induit par $u'$, est une $\W$\nobreakdash-équivalence localement sur $B'$.
\end{GMitemize}
De plus, ces conditions sont impliquées par la condition:
\begin{GMitemize}
\item[\rm c)] le carré $\mathcal D$ est $\W$-exact;
\end{GMitemize}
et si le localisateur fondamental $\W$ n'est pas le localisateur fondamental grossier $\Wgr$, elles lui sont équivalentes.
\end{prop}

\begin{proof}
En vertu des critères (\emph{a}) et (\emph{b}) de la proposition~\ref{carcarex}, il est immédiat que si le carré $\mathcal D$ est $\W$-exact, alors les conditions (\emph{a})  et (\emph{b}) ci-dessus sont satisfaites (\cf~\ref{eqloccoloc}).
\smallbreak

Pour montrer les autres affirmations, explicitons la condition (\emph{a}) (resp. (\emph{b})). Elle signifie que pour tout objet $a$ de $A$ et tout objet $b'$ de $B'$, le foncteur
\[
\mc{a}{\bigl(\cm{A'}{b'}\bigr)}\toto\mc{a}{\bigl(\cm{A}{w(b')}\bigr)}
\qquad\hbox{(resp. \ }
\cm{\bigl(\mc{a}{A'}\bigr)}{b'}\toto\cm{\bigl(\mc{u(a)}{B'}\bigr)}{b'}\ \ ),
\]
induit par $v$ (resp. par $u'$), est une $\W$\nbd-équivalence. On remarque que les catégories $\mc{a}{\bigl(\cm{A'}{b'}\bigr)}$ et $\cm{\bigl(\mc{a}{A'}\bigr)}{b'}$ sont canoniquement isomorphes, et que la catégorie $\mc{a}{\bigl(\cm{A}{w(b')}\bigr)}$ est non vide si et seulement si $\cm{\bigl(\mc{u(a)}{B'}\bigr)}{b'}$ l'est (dans les deux cas cela signifie que $\Hom^{}_B(u(a),w(b'))\neq\varnothing$), ce qui prouve l'équivalence des conditions (\emph{a}) et (\emph{b}) dans le cas du localisateur grossier $\Wgr$ (\cf~\ref{defWgr}). Pour achever la preuve de la proposition, il suffit donc de montrer que si $\W\neq\Wgr$, alors la condition (\emph{a}) implique la condition (\emph{c}), ce qui prouvera l'équivalence de ces deux conditions, et dualement celle de (\emph{b}) et (\emph{c}). Comme l'implication (\emph{a}) $\Rightarrow$ (\emph{c}) est évidente pour le localisateur trivial $\Wtr=\Fl\,\Cat$, il suffit de la montrer en supposant que $\W\subset\Wzer$ (\cf~\ref{defWgr}). Pour cela, on remarque que le foncteur $\mc{a}{\bigl(\cm{A'}{b'}\bigr)}\toto\mc{a}{\bigl(\cm{A}{w(b')}\bigr)}$ s'identifie canoniquement au foncteur somme
\[
\textstyle\coprod\limits_{g:u(a)\toto w(b')}\mc{(a,g)}{\bigl(\cm{A'}{b'}\bigr)}\hskip 8pt\Toto{2}{}
\textstyle\coprod\limits_{g:u(a)\toto w(b')}\mc{(a,g)}{\bigl(\cm{A}{w(b')}\bigr)}
\]
(où le couple $(a,g)$ est vu comme objet de $\cm{A}{w(b')}$). Si ce foncteur est une $\W$\nbd-équivalence, il est à plus forte raison une $\Wzer$\nbd-équivalence, et en particulier pour toute flèche $g:u(a)\toto w(b')$, la catégorie $\mc{(a,g)}{\bigl(\cm{A'}{b'}\bigr)}$ est non vide. Par suite, le foncteur $\mc{(a,g)}{\bigl(\cm{A'}{b'}\bigr)}\toto\mc{(a,g)}{\bigl(\cm{A}{w(b')}\bigr)}$ est un rétracte du foncteur somme, et est donc une $\W$-équivalence (\cf~\ref{stablocfond}). On en déduit que le foncteur $\cm{A'}{b'}\toto\cm{A}{w(b')}$ est $\W$\nbd-coasphérique, ce qui implique, en vertu du critère (\emph{a}) de la proposition~\ref{carcarex}, que le carré $\mathcal D$ est $\W$\nbd-exact, et achève la démonstration.
\end{proof}

\begin{paragr} \titparagr{Carrés exacts faibles} Un \emph{carré $\W$\nbd-exact faible} est un carré satisfaisant aux conditions équivalentes (\emph{a}) et (\emph{b}) de la proposition précédente. En vertu de cette proposition, si le localisateur fondamental $\W$ n'est pas le localisateur fondamental grossier $\Wgr$, cette notion coïncide avec celle de carré $\W$\nbd-exact. L'intérêt de cette notion serait donc extrêmement limité si ce n'était pas elle qui apparaît naturellement dans la théorie des dérivateurs, et non \emph{pas} celle de carré exact (\cf~théorème~\ref{carcarexder}). On laisse comme exercice au lecteur d'énoncer et démontrer les analogues des propositions~\ref{carexop}, \ref{compcarhex}, \ref{desccarhexh} et \ref{changloccarex}, pour les carrés $\W$\nbd-exacts faibles. En revanche, la caractérisation des foncteurs $\W$\nbd-propres de la proposition~\ref{carpropreex} n'est plus vraie si on remplace carré $\W$\nbd-exact par carré $\W$\nbd-exact faible. Elle doit être remplacée par la caractérisation suivante:

\begin{prop} \label{carpropreexf}
Soit $u:A\toto B$ une flèche de $\Cat$. Les conditions suivantes sont équivalentes:
\begin{GMitemize}
\item[\rm a)] $u$ est $\W$-propre;
\item[\rm b)] pour tout diagramme de carrés cartésiens de la forme 
\[
\UseTwocells
\xymatrixcolsep{2.5pc}
\xymatrixrowsep{2.4pc}
\xymatrix{
&A''\ar[d]_{u''}\ar[r]^{v'}
\drtwocell<\omit>{}
&A'\ar[d]^{u'}\ar[r]^{v}\ar@<-2ex>@{}[d]_(.55){=}
\drtwocell<\omit>{}
&A\ar[d]^{u}\ar@<-2ex>@{}[d]_(.55){=}
\\
&B''\ar[r]_{w'}
&B'\ar[r]_{w}
&B
&\hskip -30pt,
}
\]
le carré de gauche est $\W$-exact faible. 
\end{GMitemize}
\end{prop}

\begin{proof} L'implication (\emph{a}) $\Rightarrow$ (\emph{b}) résulte des propositions~\ref{proprelisseex} et \ref{carcarexf}, et de la stabilité des foncteurs propres par changement de base~\cite[corollaire 3.2.4]{Ast}. Pour prouver l'implication (\emph{b}) $\Rightarrow$ (\emph{a}), soit $b$ un objet de $B$, et considérons le diagramme de carrés cartésiens 
\[
\UseTwocells
\xymatrixcolsep{2.5pc}
\xymatrixrowsep{2.4pc}
\xymatrix{
&A_b\ar[d]_{}\ar[r]^{}
\drtwocell<\omit>{}
&\cm{A}{b}\ar[d]^{}\ar[r]^{}\ar@<-2ex>@{}[d]_(.55){=}
\drtwocell<\omit>{}
&A\ar[d]^{u}\ar@<-2ex>@{}[d]_(.55){=}
\\
&e\ar[r]_{(b,\id{b})}
&\cm{B}{b}\ar[r]_{}
&B
&\hskip -30pt,
}
\]
où $\cm{B}{b}\toto B$ est le foncteur d'oubli, et la flèche $e\toto\cm{B}{b}$ est définie par l'objet final $(b,\id{b})$ de $\cm{B}{b}$. En vertu de la condition (\emph{b}), le carré de gauche est $\W$\nbd-exact faible, et il résulte alors du critère (\emph{a}) de la proposition~\ref{carcarexf} que la flèche $A_b\toto\cm{\bigl(\cm{A}{b}\bigr)}{(b,\id{b})}\simeq\cm{A}{b}$ est une $\W$\nbd-équivalence colocalement sur $\cm{A}{b}$, autrement dit qu'elle est $\W$\nbd-coasphérique, ce qui prouve que $u$ est $\W$\nbd-propre.
\end{proof}
\end{paragr}

\begin{ex} Pour le localisateur fondamental grossier $\W=\Wgr$ (\cf~\ref{defWgr}), les notions de carré $\W$-exact et carré $\W$\nbd-exact faible sont bien distinctes. En effet, considérons le carré
\[
\UseAllTwocells
\xymatrixcolsep{.0pc}
\xymatrixrowsep{.3pc}
\xymatrix{
A'=
&e\ar[dddd]_{}\ar[rrrrrr]^{}
&&&&&&e\ar[ddd]^{0}
&=A
\\&&&&
\ddrrcompositemap<\omit>{\alpha}
&&&
\\&&&&&&&
\\&&&&&&&
\\
B'=
&e\ar[rrrrr]_-{1}
&&&&&\{0\ar@/^.5pc/[rr]^{\alpha}\ar@/_.5pc/[rr]_{\beta}
&{\phantom{\}}}
&1\}
&\kern -5pt=B
&.
}
\]
Si $b'$ désigne l'unique objet de $B'=e$, le foncteur $\cm{A'}{b'}\toto\cm{A}{1}$, induit par la flèche horizontale du haut, s'identifie au morphisme
\[
\cm{A'}{b'}\simeq e\Toto{2}{\alpha}\{\alpha,\beta\}\simeq \cm{A}{1}
\]
de la catégorie ponctuelle $e$ vers la catégorie discrète ayant comme objets $\alpha$ et $\beta$, défini par l'objet $\alpha$. Ce morphisme est une $\W$\nbd-équivalence colocalement sur $A=e$, autrement dit, une $\W$\nbd-équivalence, puisque les catégories source et but sont toutes deux non vides. En revanche, il n'est \emph{pas} $\W$\nbd-coasphérique, puisque la catégorie $\mc{\beta}{(\cm{A'}{b'})}$ est vide. Ainsi, le carré $\mathcal D$ est $\W$\nbd-exact faible, mais \emph{pas} $\W$\nbd-exact.
\end{ex}

\section{\boldmath Structures de carrés exacts sur $\Cat$}

\begin{paragr} \titparagr{Classes de carrés} \label{classcar}
Dans cette section, on va considérer des classes $\Q$ de carrés dans $\Cat$, des propriétés de stabilité de telles classes, ainsi que des conditions de non trivialité, ou \og d'initialisation\fg{}, affirmant que les carrés d'un type donné appartiennent à $\Q$. \'Etant donnée une telle classe $\Q$, on dira qu'un carré de $\Cat$ est $\Q$\nbd-\emph{exact} s'il appartient à cette classe. On notera $\op{\Q}$ la classe formée des carrés obtenus par passage aux catégories opposées à partir de carrés appartenant à $\Q$. Toutes les classes $\Q$ considérées seront stables par isomorphisme de carrés.
\end{paragr}

\begin{paragr} \titparagr{Propriétés de stabilité de classes de carrés} \label{classstab}
Soit $\Q$ une classe de carrés de $\Cat$. Les conditions de stabilité les plus importantes qu'on aura à considérer sont les suivantes:
\smallskip

\noindent
\textbf{\boldmath CS\,$1_{\mathrm h}$ \ (\emph{Composition horizontale.})} La classe $\Q$ est stable par composition horizontale: 
Considérons deux carrés composables horizontalement et leur composé
\[
\xymatrixrowsep{1.1pc}
\xymatrixcolsep{.9pc}
\UseAllTwocells
\xymatrix{
A''\ddrrcompositemap<\omit>{\alpha'}
  \ar[rr]^{v'}
  \ar[dd]_{u''}
&&A'\ddrrcompositemap<\omit>{\alpha}
  \ar[rr]^{v}
  \ar[dd]^{u'}
&&A\ar[dd]^u
&&&A''\ddrrcompositemap<\omit>{\alpha''}
  \ar[rr]^{v''=vv'}
  \ar[dd]_{u''}
&&A\ar[dd]^u
\\
\\
B''\ar[rr]_{w'}
&&B'\ar[rr]_w
&&B
&,
&&B''\ar[rr]_{w''=ww'\,\,}
&&B
&,
\\
&\mathcal D'
&&\mathcal D
&&&&&\kern -10pt \mathcal D\circh\mathcal D'\kern -10pt
}
\]
où $\alpha'' =(w\star\alpha')(\alpha\star v')$. Si $\mathcal D$ et $\mathcal D'$ sont $\Q$\nbd-exacts, il en est de même de $\mathcal D\circh\mathcal D'$.
\smallbreak

\noindent
\textbf{\boldmath CS\,$1_{\mathrm v}$ \ (\emph{Composition verticale.})} La classe $\Q$ est stable par composition verticale: 
Considérons deux carrés composables verticalement et leur composé
$$
\xymatrixrowsep{.8pc}
\xymatrixcolsep{1.4pc}
\UseAllTwocells
\xymatrix{
&A'\ar[rr]^f\ar[dd]_{u'}\ddrrcompositemap<\omit>{\alpha}
&&A\ar[dd]^u
&&&&A'\ar[rr]^f\ar[dd]_{v'u'}\ddrrcompositemap<\omit>{\,\gamma}
&&A\ar[dd]^{vu}
\\
{\mathcal D}\quad
\\
&B'\ar[rr]^g\ar[dd]_{v'}\ddrrcompositemap<\omit>{\beta}
&&B\ar[dd]^v
&&&&C'\ar[rr]_h
&&C
&\hskip -15pt,
\\
{\mathcal E}\quad
&&&&&&&&\hskip -30pt{\mathcal E}\circv\mathcal D\hskip -30pt
\\
&C'\ar[rr]_h
&&C
&\hskip -15pt,
}
$$
où $\gamma =(\beta\star u')(v\star\alpha)$. Si $\mathcal D$ et $\mathcal E$ sont $\Q$\nbd-exacts, il en est de même de $\mathcal E\circv\mathcal D$.
\smallbreak

\noindent
\textbf{\boldmath CS\,$2_{\mathrm h}$ \ (\emph{Descente horizontale.})} Soient $J$ un ensemble, et 
\[
\xymatrixrowsep{.7pc}
\xymatrixcolsep{1.3pc}
\UseAllTwocells
\xymatrix{
&A'\ddrrcompositemap<\omit>{\alpha}
  \ar[rr]^{v}
  \ar[dd]_{u'}
&&A\ar[dd]^u
&&&A_j\ddrrcompositemap<\omit>{\ \alpha_j}
  \ar[rr]^{v_j}
  \ar[dd]_{u_j}
&&A'\ar[dd]^{u'}
\\
{\mathcal D}\ =
&&&&\hbox{et}
&{\mathcal D_j}\ =
&&&&,\ j\in J\ ,
\\
&B'\ar[rr]_w
&&B
&&&B_j\ar[rr]_{w_j}
&&B'
}
\]
des carrés dans $\Cat$ tels que $\Ob\,B'=\bigcup\limits_{j\in J}w_j(\Ob\,B_j)$. Si pour tout élément $j$ de $J$, les carrés $\mathcal D_j$ et $\mathcal D\circh\mathcal D_j$ sont $\Q$\nbd-exacts, il en est de même de $\mathcal D$.
\smallbreak

\noindent
\textbf{\boldmath CS\,$2_{\mathrm v}$ \ (\emph{Descente verticale.})} Soient $J$ un ensemble, et 
\[
\xymatrixrowsep{.7pc}
\xymatrixcolsep{1.3pc}
\UseAllTwocells
\xymatrix{
&A'\ddrrcompositemap<\omit>{\alpha}
  \ar[rr]^{v}
  \ar[dd]_{u'}
&&A\ar[dd]^u
&&&A'_j\ddrrcompositemap<\omit>{\ \alpha_j}
  \ar[rr]^{v_j}
  \ar[dd]_{u'_j}
&&A_j\ar[dd]^{u_j}
\\
{\mathcal D}\ =
&&&&\hbox{et}
&{\mathcal D_j}\ =
&&&&,\ j\in J\ ,
\\
&B'\ar[rr]_w
&&B
&&&A'\ar[rr]_{v}
&&A
}
\]
des carrés dans $\Cat$ tels que $\Ob\,A=\bigcup\limits_{j\in J}u_j(\Ob\,A_j)$. Si pour tout élément $j$ de $J$, les carrés $\mathcal D_j$ et $\mathcal D\circv\mathcal D_j$ sont $\Q$\nbd-exacts, il en est de même de $\mathcal D$.
\smallbreak

\noindent
\textbf{\boldmath CS\,$3_{\mathrm h}$ \ (\emph{Descente locale.})} Soit 
\[
\mathcal{D}\ =\quad
\raise 23pt
\vbox{
\UseTwocells
\xymatrixcolsep{2.5pc}
\xymatrixrowsep{2.4pc}
\xymatrix{
A'\ar[d]_{u'}\ar[r]^{v}
\drtwocell<\omit>{\alpha}
&A\ar[d]^{u}
\\
B'\ar[r]_{w}
&B
&\hskip -30pt
}
}
\] 
un carré dans $\Cat$. Si pour tout objet $b'$ de $B'\!$, le composé horizontal $\mathcal D\circh \mathcal{D}^{\mathrm{\,g}}_{u',b'}$, où $\mathcal{D}^{\mathrm{\,g}}_{u',b'}$ désigne le carré comma
\[
\hskip 15pt
\mathcal{D}_{u',b'}^{\mathrm{\,g}}\ =\quad
\raise 23pt
\vbox{
\UseTwocells
\xymatrixcolsep{2.5pc}
\xymatrixrowsep{2.4pc}
\xymatrix{
\cm{A'}{b'}\ar[d]_{}\ar[r]^{}
\drtwocell<\omit>{}
&A'\ar[d]^{u'}
\\
e\ar[r]_{b'}
&B'
&\hskip -30pt,
}
}\qquad
\]
est $\Q$-exact, il en est de même de $\mathcal D$.
\smallbreak

\noindent
\textbf{\boldmath CS\,$3_{\mathrm v}$ \ (\emph{Descente colocale.})} Soit 
\[
\mathcal{D}\ =\quad
\raise 23pt
\vbox{
\UseTwocells
\xymatrixcolsep{2.5pc}
\xymatrixrowsep{2.4pc}
\xymatrix{
A'\ar[d]_{u'}\ar[r]^{v}
\drtwocell<\omit>{\alpha}
&A\ar[d]^{u}
\\
B'\ar[r]_{w}
&B
&\hskip -30pt
}
}
\] 
un carré dans $\Cat$. Si pour tout objet $a$ de $A$, le composé vertical $\mathcal D\circv \mathcal{D}^{\mathrm{\,d}}_{v,a}$, où $\mathcal{D}^{\mathrm{\,d}}_{v,a}$ désigne le carré comma
\[
\hskip 15pt
\mathcal{D}_{v,a}^{\mathrm{\,d}}\ =\quad
\raise 23pt
\vbox{
\UseTwocells
\xymatrixcolsep{2.5pc}
\xymatrixrowsep{2.4pc}
\xymatrix{
\mc{a}{A'}\ar[d]_{}\ar[r]^{}
\drtwocell<\omit>{}
&e\ar[d]^{a}
\\
A'\ar[r]_{v}
&A
&\hskip -30pt,
}
}\qquad
\]
est $\Q$-exact, il en est de même de $\mathcal D$.
\smallbreak

\noindent
\textbf{CS\,4 \ (\emph{Passage à l'opposé.})} Si le carré $\mathcal D$ de $\Cat$ est $\Q$\nbd-exact, il en est de même du carré $\op{\mathcal D}$, obtenu par passage aux catégories opposées, ce qui implique que $\op{\Q}=\Q$.
\end{paragr}

\begin{paragr} \titparagr{Propriétés \og d'initialisation\fg} \label{classinit}
Les conditions du paragraphe précédent étant des conditions de stabilité, elles sont toutes satisfaites par la classe vide. Pour obtenir des classes $\Q$ non triviales, il faut imposer qu'elles contiennent certains types de carrés. Parmi ces conditions \og d'initialisation\fg{}, les plus importantes envisagées sont énumérées ci-dessous. Certaines de ses conditions sont \og absolues\fg, et certaines dépendent de la donnée d'un localisateur fondamental faible $\W$, donné une fois pour toutes. On rappelle que $e$ désigne la catégorie ponctuelle.
\smallbreak

\noindent
\textbf{CI\,1g} \ Si $u:A\toto B$ est un foncteur $\W$\nbd-asphérique, alors le carré
\[
\UseTwocells
\xymatrixcolsep{2.5pc}
\xymatrixrowsep{2.4pc}
\xymatrix{
A\ar[d]_{u}\ar[r]^{}
\drtwocell<\omit>{}
&e\ar[d]^{}\ar@<-2ex>@{}[d]_(.55){=}\ar@{}[d]
\\
B\ar[r]_{}
&e
}
\]
est $\Q$\nbd-exact.
\smallbreak

\noindent
\textbf{CI\,1d} \ Si $u:A\toto B$ est un foncteur $\W$\nbd-coasphérique, alors le carré
\[
\UseTwocells
\xymatrixcolsep{2.5pc}
\xymatrixrowsep{2.4pc}
\xymatrix{
A\ar[d]_{}\ar[r]^{u}
\drtwocell<\omit>{}
&B\ar[d]\ar@<-2ex>@{}[d]_(.55){=}
\\
e\ar[r]_{}
&e
}
\]
est $\Q$\nbd-exact.
\smallbreak

\noindent
\textbf{\boldmath CI\,$1_0$} \ Si $A$ est une catégorie $\W$\nbd-asphérique, alors le carré
\[
\UseTwocells
\xymatrixcolsep{2.5pc}
\xymatrixrowsep{2.4pc}
\xymatrix{
A\ar[d]_{}\ar[r]^{}
\drtwocell<\omit>{}
&e\ar[d]^{}\ar@<-2ex>@{}[d]_(.55){=}
\\
e\ar[r]_{}
&e
}
\]
est $\Q$\nbd-exact.
\smallbreak

\noindent
\textbf{\boldmath CI\,$1'$g} \ Si $u:A\toto B$ est un foncteur admettant un adjoint à droite, alors le carré
\[
\UseTwocells
\xymatrixcolsep{2.5pc}
\xymatrixrowsep{2.4pc}
\xymatrix{
A\ar[d]_{u}\ar[r]^{}
\drtwocell<\omit>{}
&e\ar[d]^{}\ar@<-2ex>@{}[d]_(.55){=}\ar@{}[d]
\\
B\ar[r]_{}
&e
}
\]
est $\Q$\nbd-exact.
\smallbreak

\noindent
\textbf{\boldmath CI\,$1'$d} \ Si $u:A\toto B$ est un foncteur admettant un adjoint à gauche, alors le carré
\[
\UseTwocells
\xymatrixcolsep{2.5pc}
\xymatrixrowsep{2.4pc}
\xymatrix{
A\ar[d]_{}\ar[r]^{u}
\drtwocell<\omit>{}
&B\ar[d]\ar@<-2ex>@{}[d]_(.55){=}
\\
e\ar[r]_{}
&e
}
\]
est $\Q$\nbd-exact.
\smallbreak

On observe qu'on a le diagramme d'implications tautologiques suivant (\cf~\ref{asphcoasph})
\[
\xymatrixcolsep{.5pc}
\xymatrixrowsep{1.4pc}
\xymatrix{
&\hbox{\textbf{CI\,1g}}\ar@{=>}[ld]\ar@{=>}[rd]
&&\hbox{\textbf{CI\,1d}}\ar@{=>}[ld]\ar@{=>}[rd]
\\
\hbox{\textbf{\boldmath CI\,$1'$g}}
&&\textbf{\boldmath CI\,$1_0$}
&&\hbox{\textbf{\boldmath CI\,$1'$d}}
&.
}
\]
\smallbreak

\noindent
\textbf{CI\,2g} \ Pour toute flèche $u:A\toto B$ de $\Cat$, et tout objet $b$ de $B$, le carré comma
\[
\mathcal{D}_{u,b}^{\mathrm{\,g}}\ =\quad
\raise 23pt
\vbox{
\UseTwocells
\xymatrixcolsep{2.5pc}
\xymatrixrowsep{2.4pc}
\xymatrix{
\cm{A}{b}\ar[d]_{p}\ar[r]^{j}
\drtwocell<\omit>{\alpha}
&A\ar[d]^{u}
\\
e\ar[r]_{b}
&B
}
}\qquad
\]
(où $b:e\toto B$ désigne la flèche définie par l'objet $b$ de $B$) est $\Q$\nbd-exact.
\smallbreak

\noindent
\textbf{CI\,2d} \ Pour toute flèche $u:A\toto B$ de $\Cat$, et tout objet $b$ de $B$, le carré comma
\[
\mathcal{D}_{u,b}^{\mathrm{\,d}}\ =\quad
\raise 23pt
\vbox{
\UseTwocells
\xymatrixcolsep{2.5pc}
\xymatrixrowsep{2.4pc}
\xymatrix{
\mc{b}{A}\ar[d]_{k}\ar[r]^{q}
\drtwocell<\omit>{\beta}
&e\ar[d]^{b}
\\
A\ar[r]_{u}
&B
}
}\qquad
\]
(où $b:e\toto B$ désigne la flèche définie par l'objet $b$ de $B$) est $\Q$\nbd-exact.
\smallbreak

On observe qu'en présence de \textbf{CI\,2g} (resp. de \textbf{CI\,2d}) la condition de descente horizontale (resp. verticale) implique la condition de descente locale (resp. colocale):
\begin{equation} \label{implev}
\hbox{(\textbf{CI\,2g} et \textbf{\boldmath CS\,$2_{\mathrm h}$})}\Longrightarrow\hbox{\textbf{\boldmath CS\,$3_{\mathrm h}$}}\smsp,\qquad
\hbox{(\textbf{CI\,2d} et \textbf{\boldmath CS\,$2_{\mathrm v}$})}\Longrightarrow\hbox{\textbf{\boldmath CS\,$3_{\mathrm v}$}}\smsp.
\end{equation}
\end{paragr}

\begin{paragr} \titparagr{Exemples de classes de carrés} \label{classex}
On s'intéresse plus particulièrement aux classes de carrés suivantes.
\smallskip

\begin{subparagr} \titsubparagr{La classe des carrés exacts homotopiques} \label{classcarex}
On a vu dans la section précédente que pour tout localisateur fondamental faible $\W$, la classe $\Q$ des carrés $\W$\nbd-exacts satisfait à \emph{toutes} les conditions de stabilité et d'initialisation considérées dans les paragraphes~\ref{classstab} et~\ref{classinit} 
(\cf~\ref{extrcarex} -- \ref{carexcomma}).
\end{subparagr}

\begin{subparagr} \titsubparagr{La classe des carrés exacts homotopiques faibles} \label{classcarexf}
Pour tout localisateur fondamental $\W$, la classe $\Q$ des carrés $\W$\nbd-exacts faibles satisfait aussi à toutes les conditions considérées dans les paragraphes~\ref{classstab} et~\ref{classinit}.
\end{subparagr}

\begin{subparagr} \titsubparagr{La classe des carrés de Beck-Chevalley à gauche} \label{classcarBCg} On vérifie facilement que la classe $\Q$ des carrés de Beck-Chevalley à gauche est stable par composition horizontale \emph{et} verticale (conditions~\textbf{\boldmath CS\,$1_{\mathrm h}$} et~\textbf{\boldmath CS\,$1_{\mathrm v}$}), et satisfait tautologiquement à la condition \textbf{\boldmath CI\,$1'$g}. 
\end{subparagr}

\begin{subparagr} \titsubparagr{La classe des carrés de Beck-Chevalley à droite} \label{classcarBCd} Dualement, la classe $\Q$ des carrés de Beck-Chevalley à droite est stable par composition horizontale \emph{et} verticale (conditions~\textbf{\boldmath CS\,$1_{\mathrm h}$} et~\textbf{\boldmath CS\,$1_{\mathrm v}$}), et satisfait à la condition~\textbf{\boldmath CI\,$1'$d}. 
\end{subparagr}

\begin{subparagr} \titsubparagr{La classe des carrés de Beck-Chevalley} \label{classcarBCgd} La classe $\Q$ des carrés de Beck-Chevalley à gauche \emph{et} à droite est stable par composition horizontale et verticale, et par passage aux catégories opposées (conditions~\textbf{\boldmath CS\,$1_{\mathrm h}$}, \textbf{\boldmath CS\,$1_{\mathrm v}$} et~\textbf{CS\,4}). 
\end{subparagr}

\begin{subparagr} \titsubparagr{La classe des carrés comma} \label{classcommacar} La classe $\Q$ des carrés comma ne satisfait que la condition de stabilité par passage aux catégories opposées (condition~\textbf{CS\,4}), et tautologiquement les conditions \textbf{CI\,2g} et \textbf{CI\,2d}.
\end{subparagr}
\end{paragr}

\begin{lemme} \label{condsuffex1}
Soient $\W$ un localisateur fondamental faible, et $\Q$ une classe de carrés de $\Cat$ satisfaisant aux conditions \emph{\textbf{\boldmath CS\,$1_{\mathrm h}$}, \textbf{\boldmath CS\,$3_{\mathrm h}$}, \textbf{CI\,1d}} et \emph{\textbf{CI\,2g}}. Alors $\Q$ contient la classe des carrés $\W$\nbd-exacts.
\end{lemme}

\begin{proof}
Soient donc  
\[
\mathcal{D}\ =\quad
\raise 23pt
\vbox{
\UseTwocells
\xymatrixcolsep{2.5pc}
\xymatrixrowsep{2.4pc}
\xymatrix{
A'\ar[d]_{u'}\ar[r]^{v}
\drtwocell<\omit>{\alpha}
&A\ar[d]^{u}
\\
B'\ar[r]_{w}
&B
&\hskip -30pt
}
}
\]
un carré $\W$-exact, $b'$ un objet de $B'$, et considérons les diagrammes
\[
\xymatrixrowsep{1.pc}
\xymatrixcolsep{.4pc}
\UseAllTwocells
\xymatrix{
\cm{A'}{b'}\ddrrcompositemap<\omit>{}
  \ar[rr]^{}
  \ar[dd]_{}
&&A'\ddrrcompositemap<\omit>{\alpha}
  \ar[rr]^{v}
  \ar[dd]_{u'}
&&A\ar[dd]^{u\hskip 20pt\hbox{et}}
\\
\\
e\ar[rr]_{b'}
&&B'\ar[rr]_w
&&B
\\
&\hskip -5pt\mathcal D^{\mathrm{\,g}}_{u',b'}\hskip -5pt
&&\hskip 6pt\mathcal D\hskip 6pt
}
\hskip 10pt
\xymatrix{
\cm{A'}{b'}\ddrrcompositemap<\omit>{}
  \ar[rr]^{}
  \ar[dd]_{}
&&\cm{A}{w(b')}\ddrrcompositemap<\omit>{}
  \ar[rr]^{}
  \ar[dd]_{}\ar@<-3.3ex>@{}[dd]_(.55){=}
&&A\ar[dd]^{u}
\\
\\
e\ar[rr]_{}
&&e\ar[rr]_{w(b')}
&&B
&\hskip 5pt.
\\
&\hskip -2pt\mathcal D'\hskip -4pt
&&\hskip -13pt\mathcal D^{\mathrm{\,g}}_{u,w(b')}\hskip -11pt
}
\]
On observe qu'on a l'égalité
\[
\mathcal D\,\circh\,\mathcal D^{\mathrm{\,g}}_{u',b'}\,=\,\mathcal D^{\mathrm{\,g}}_{u,w(b')}\,\circh\,\mathcal D'\smsp.
\]
Comme le carré $\mathcal D$ est $\W$\nbd-exact, en vertu du critère (\emph{a}) de la proposition~\ref{carcarex}, la flèche $\cm{A'}{b'}\toto\cm{A}{w(b')}$ est $\W$\nbd-coasphérique. Il résulte donc de la condition~\textbf{CI\,1d} que le carré $\mathcal D'$ est $\Q$\nbd-exact. D'autre part, la condition~\textbf{CI\,2g} implique que les carrés comma $\mathcal D^{\mathrm{\,g}}_{u,w(b')}$ et $\mathcal D^{\mathrm{\,g}}_{u',b'}$ sont $\Q$\nbd-exacts. En vertu de la condition~\textbf{\boldmath CS\,$1_{\mathrm h}$}, le carré composé $\mathcal D^{\mathrm{\,g}}_{u,w(b')}\,\circh\,\mathcal D'$ est donc aussi $\Q$\nbd-exact. L'égalité ci-dessus, et la condition~\textbf{\boldmath CS\,$3_{\mathrm h}$} impliquent alors que le carré $\W$\nbd-exact $\mathcal D$ appartient à la classe $\Q$.
\end{proof}

\begin{thm} \label{carclassex1}
Soit $\W$ un localisateur fondamental faible. La classe des carrés $\W$-exacts est la plus petite classe de carrés, stable par composition horizontale, satisfaisant à la condition de descente locale et contenant les carrés comma, ainsi que les carrés de la forme
\[
\UseTwocells
\xymatrixcolsep{2.5pc}
\xymatrixrowsep{2.4pc}
\xymatrix{
&A\ar[d]_{}\ar[r]^{u}
\drtwocell<\omit>{}
&B\ar[d]\ar@<-2ex>@{}[d]_(.55){=}
\\
&e\ar[r]_{}
&e
&\hskip -30pt,
}
\]
avec $u$ morphisme $\W$-coasphérique de $\Cat$.
\end{thm}

\begin{proof}
Le théorème est conséquence immédiate du lemme précédent, des propositions~\ref{compcarhex}, \ref{desccarhexh}, \ref{carexcomma}, et de l'exemple~\ref{extrcarex}.
\end{proof}

\begin{lemme} \label{consuffexcoasph}
Soient $\W$ un localisateur fondamental faible, et $\Q$ une classe de carrés de $\Cat$ satisfaisant aux conditions \emph{\textbf{\boldmath CS\,$3_{\mathrm v}$}} et \emph{\textbf{\boldmath CI\,$1_0$}}. Alors $\Q$ satisfait aussi à la condition \emph{\textbf{CI\,1d}}.
\end{lemme}

\begin{proof}
Soient $u:A\toto B$ un morphisme $\W$\nbd-coasphérique de $\Cat$, $b$ un objet de $B$, et considérons le diagramme
\[
\UseTwocells
\xymatrixcolsep{2.5pc}
\xymatrixrowsep{2.4pc}
\xymatrix{
&\mc{b}{A}\ar[d]\ar[r]
\drtwocell<\omit>{}
&e\ar[d]
\\
&A\ar[d]_{}\ar[r]^{u}
\drtwocell<\omit>{}
&B\ar[d]\ar@<-2ex>@{}[d]_(.55){=}
\\
&e\ar[r]_{}
&e
&\hskip -30pt.
}
\]
Comme la catégorie $\mc{b}{A}$ est $\W$\nbd-asphérique, la condition~\textbf{\boldmath CI\,$1_0$} implique que le carré composé est $\Q$\nbd-exact, et par suite, en vertu de la condition de descente colocale~\textbf{\boldmath CS\,$3_{\mathrm v}$}, il en est de même du carré du bas, ce qui prouve la condition~\textbf{CI\,1d}.
\end{proof}

\begin{lemme} \label{condsuffex2}
Soient $\W$ un localisateur fondamental faible, et $\Q$ une classe de carrés de $\Cat$ satisfaisant aux conditions \emph{\textbf{\boldmath CS\,$1_{\mathrm h}$}, \textbf{\boldmath CS\,$3_{\mathrm h}$}, \textbf{\boldmath CS\,$3_{\mathrm v}$}, \textbf{\boldmath CI\,$1_0$}} et \emph{\textbf{CI\,2g}}. Alors $\Q$ contient la classe des carrés $\W$\nbd-exacts.
\end{lemme}

\begin{proof}
Le lemme s'obtient en combinant les lemmes~\ref{condsuffex1} et~\ref{consuffexcoasph}.
\end{proof}

\begin{thm} \label{carclassex2}
Soit $\W$ un localisateur fondamental faible. La classe des carrés $\W$-exacts est la plus petite classe de carrés stable par passage aux catégories opposées, par composition horizontale, satisfaisant à la condition de descente locale et contenant les carrés comma, ainsi que les carrés de la forme
\[
\UseTwocells
\xymatrixcolsep{2.5pc}
\xymatrixrowsep{2.4pc}
\xymatrix{
&A\ar[d]_{}\ar[r]^{}
\drtwocell<\omit>{}
&e\ar[d]\ar@<-2ex>@{}[d]_(.55){=}
\\
&e\ar[r]_{}
&e
&\hskip -30pt,
}
\]
avec $A$ une petite catégorie $\W$-asphérique.
\end{thm}

\begin{proof} 
Le théorème est conséquence immédiate du lemme précédent, des propositions~\ref{carexop}, \ref{compcarhex}, \ref{desccarhexh}, \ref{carexcomma}, et de l'exemple~\ref{extrcarex}.
\end{proof}

\begin{cor} \label{carclassexGuit}
La classe des carrés exacts de Guitart est la plus petite classe de carrés stable par passage aux catégories opposées, par composition horizontale, satisfaisant à la condition de descente locale et contenant les carrés comma, ainsi que les carrés de la forme
\[
\UseTwocells
\xymatrixcolsep{2.5pc}
\xymatrixrowsep{2.4pc}
\xymatrix{
&A\ar[d]_{}\ar[r]^{}
\drtwocell<\omit>{}
&e\ar[d]\ar@<-2ex>@{}[d]_(.55){=}
\\
&e\ar[r]_{}
&e
&\hskip -30pt,
}
\]
avec $A$ une petite catégorie $0$-connexe.
\end{cor}

\begin{proof}
C'est le cas particulier du théorème, appliqué au localisateur fondamental $\Wzer$.
\end{proof}

\noindent
Les lemmes suivants seront utiles à la théorie des dérivateurs dans la section suivante.

\begin{lemme} \label{condsuffexcomma}
Toute classe $\Q$ de carrés satisfaisant aux conditions \emph{\textbf{\boldmath CS\,$1_{\mathrm h}$}, \textbf{\boldmath CS\,$2_{\mathrm h}$}, \textbf{\boldmath CI\,$1'$d}} et \emph{\textbf{CI\,2g}} contient la classe des carrés comma.
\end{lemme}

\begin{proof}
Soit
\[
\mathcal{D}\ =\quad
\raise 23pt
\vbox{
\UseTwocells
\xymatrixcolsep{2.5pc}
\xymatrixrowsep{2.4pc}
\xymatrix{
A'\ar[d]_{u'}\ar[r]^{v}
\drtwocell<\omit>{\alpha}
&A\ar[d]^{u}
\\
B'\ar[r]_{w}
&B
&\hskip -30pt,
}
}
\]
un carré comma. On remarque que la catégorie $A'$ s'identifie à la catégorie cofibrée sur $B'$, définie par le foncteur
\[
B'\toto\Cat\smsp,\qquad b'\mapstoto\cm{A}{w(b')}\smsp,
\]
(et aussi, dualement, à la catégorie fibrée sur $A$, définie par le foncteur $\op{A}\toto\Cat$, $a\mapstoto\mc{u(a)}{B'}$), et en particulier pour tout objet $b'$ de $B'$, le foncteur canonique $\cm{A}{w(b')}\toto\cm{A'}{b'}$, de la fibre vers la catégorie comma, admet un adjoint à gauche (qui n'est autre que le foncteur $\cm{A'}{b'}\toto\cm{A}{w(b')}$, induit par $v$). Considérons le diagramme:
\[
\xymatrixrowsep{1.1pc}
\xymatrixcolsep{.9pc}
\UseAllTwocells
\xymatrix{
\cm{A}{w(b')}\ddrrcompositemap<\omit>{}
  \ar[rr]^{}
  \ar[dd]_{}
&&\cm{A'}{b'}\ddrrcompositemap<\omit>{}
  \ar[rr]^{}
  \ar[dd]^(.55){\hskip -30pt =}
&&A'\ar[dd]^{u'}
  \ddrrcompositemap<\omit>{\alpha}
  \ar[rr]^{v}
&&A\ar[dd]^u
\\
\\
e\ar[rr]_{}
&&e\ar[rr]_{b'}
&&B'\ar[rr]_{w}
&&B
&.
\\
&\kern -9pt\mathcal D'\kern -9pt
&&\kern -8pt \mathcal D^{\mathrm{\,g}}_{u',b'}\kern -8pt
&&\kern 3pt\mathcal D \kern 3pt
}
\]
On vérifie facilement que le carré composé $\mathcal D\circh\mathcal D^{\mathrm{\,g}}_{u',b'}\,\circh\,\mathcal D'$ n'est autre que le carré comma $\mathcal D^{\mathrm{\,g}}_{u,w(b')}$, défini par le foncteur $u$ et l'objet $w(b')$ de $B$. En vertu de la condition~\textbf{CI\,2g}, les carrés $\mathcal D^{\mathrm{\,g}}_{u',b'}$ et $\mathcal D^{\mathrm{\,g}}_{u,w(b')}$ sont $\Q$\nbd-exacts, et il en est de même du carré $\mathcal D'$, grâce à la condition~\textbf{\boldmath CI\,$1'$d}. La condition~\textbf{\boldmath CS\,$1_{\mathrm h}$} implique alors que le composé $\mathcal D^{\mathrm{\,g}}_{u',b'}\,\circh\,\mathcal D'$ est $\Q$\nbd-exact, et par suite, il résulte de la condition~\textbf{\boldmath CS\,$2_{\mathrm h}$} que le carré $\mathcal D$ est aussi $\Q$\nbd-exact.
\end{proof}


\begin{lemme} \label{condsuffexpropre}
Soit $\W$ un localisateur fondamental faible.
Toute classe $\Q$ de carrés satisfaisant aux conditions \emph{\textbf{\boldmath CS\,$1_{\mathrm h}$}, \textbf{\boldmath CS\,$2_{\mathrm h}$}, \textbf{CI\,2g} \emph{et} \textbf{CI\,1d} (\emph{resp}. \textbf{\boldmath CI\,$1'$d})} contient la classe des carrés cartésiens
\[
\mathcal{D}\ =\quad
\raise 23pt
\vbox{
\UseTwocells
\xymatrixcolsep{2.5pc}
\xymatrixrowsep{2.4pc}
\xymatrix{
&A'\ar[d]_{u'}\ar[r]^{v}
\drtwocell<\omit>{}
&A\ar[d]^{u}\ar@<-2ex>@{}[d]_(.55){=}
\\
&B'\ar[r]_{w}
&B
&\hskip -30pt,
}
}\qquad
\]
avec $u$ un morphisme $\W$\nbd-propre \emph{(}resp. une précofibration\emph{)}.
\end{lemme}

\begin{proof}
Soient $\mathcal D$ un tel carré cartésien, $b'$ un objet de $B'\!$, et considérons les diagrammes
\[
\xymatrixrowsep{.9pc}
\xymatrixcolsep{.9pc}
\UseAllTwocells
\xymatrix{
A_{b'}\ddrrcompositemap<\omit>{}
  \ar[rr]^{}
  \ar[dd]_{}
&&\cm{A'}{b'}\ddrrcompositemap<\omit>{}
  \ar[rr]^{}
  \ar[dd]^(.55){\hskip -24pt =}
&&A'\ar[dd]^{u'}
  \ddrrcompositemap<\omit>{}
  \ar[rr]^{v}
&&A\ar[dd]^u
  \ar[dd]^(.55){\hskip -24pt =}
&&A_{w(b')}\ddrrcompositemap<\omit>{}
  \ar[rr]
  \ar[dd]
&\kern -40pt
&\cm{A}{w(b')}\ar[dd]^(.55){\hskip -30pt =}
  \ar[rr]\ddrrcompositemap<\omit>{}
&&A\ar[dd]^u
\\
&&&&&&&\hbox{\kern 5pt et\kern -10pt}
&&\kern -40pt
&&
\\
e\ar[rr]_{}
&&e\ar[rr]_{b'}
&&B'\ar[rr]_{w}
&&B
&&e\ar[rr]
&\kern -20pt
&e\ar[rr]_{w(b')}
&&B
&.
\\
&\kern -9pt\mathcal D'\kern -9pt
&&\kern -11pt \mathcal D^{\mathrm{\,g}}_{u',b'}\kern -10pt
&&\kern 3pt\kern -6pt \mathcal D \kern 3pt
&&&&\kern -16pt\mathcal D''\kern -16pt
&&\kern -17pt \mathcal D^{\mathrm{\,g}}_{u,w(b')}\kern -10pt
}
\]
Comme le carré $\mathcal D$ est cartésien, les fibres $A'_{b'}$ et $A_{w(b')}$ sont canoniquement isomorphes, et on a un isomorphisme 
\[
\mathcal D\,\circh\, \mathcal D^{\mathrm{\,g}}_{u',b'}\,\circh\, \mathcal D'\,\simeq\, D^{\mathrm{\,g}}_{u,w(b')}\,\circh\, \mathcal D''
\]
des carrés composés. En vertu de la condition~\textbf{CI\,2g}, les carrés comma $\mathcal D^{\mathrm{\,g}}_{u',b'}$ et $\mathcal D^{\mathrm{\,g}}_{u,w(b')}$ sont $\Q$\nbd-exacts. Comme le morphisme $u$ est $\W$\nbd-propre (resp.~une précofibration), il en est de même du foncteur $u'$~\cite[corollaire 3.2.4]{Ast}, et les morphismes
\[
A'_{b'}\toto\cm{A'}{b'}\qquad\hbox{et}\qquad A_{w(b')}\toto\cm{A}{w(b')}
\]
sont $\W$\nbd-coasphériques (resp. admettent un adjoint à gauche). Il résulte donc de la condition~\textbf{CI\,1d} (resp.~\textbf{\boldmath CI\,$1'$d}) que les carrés $\mathcal D'$ et $\mathcal D''$ sont $\Q$\nbd-exacts. Les conditions~\textbf{\boldmath CS\,$1_{\mathrm h}$} et~\textbf{\boldmath CS\,$2_{\mathrm h}$} impliquent alors qu'il en est de même du carré $\mathcal D$.
\end{proof}

\begin{rem}
Ce lemme et son dual fournissent une nouvelle preuve de la proposition~\ref{proprelisseex}, puisque la classe des carrés $\W$\nbd-exacts satisfait à toutes les conditions du lemme, ainsi qu'aux conditions duales (\cf~\ref{classcarex}).
\end{rem}

\begin{lemme} \label{condsuffexfibrdiscr}
Toute classe $\Q$ de carrés satisfaisant aux conditions \emph{\textbf{\boldmath CS\,$3_{\mathrm h}$} \emph{et} \textbf{CI\,2g}} contient la classe des carrés cartésiens
\[
\mathcal{D}\ =\quad
\raise 23pt
\vbox{
\UseTwocells
\xymatrixcolsep{2.5pc}
\xymatrixrowsep{2.4pc}
\xymatrix{
&A'\ar[d]_{u'}\ar[r]^{v}
\drtwocell<\omit>{}
&A\ar[d]^{u}\ar@<-2ex>@{}[d]_(.55){=}
\\
&B'\ar[r]_{w}
&B
&\hskip -30pt,
}
}\qquad
\]
avec $w$ une fibration discrète.
\end{lemme}

\begin{proof}
Soient $\mathcal D$ un tel carré cartésien, $b'$ un objet de $B'$, et considérons les diagrammes
\[
\xymatrixrowsep{.9pc}
\xymatrixcolsep{.9pc}
\UseAllTwocells
\xymatrix{
\cm{A'}{b'}\ddrrcompositemap<\omit>{}
  \ar[rr]^{}
  \ar[dd]
&&A'\ar[dd]^{u'}
  \ddrrcompositemap<\omit>{}
  \ar[rr]^{v}
&&A\ar[dd]^u
  \ar[dd]^(.55){\hskip -24pt =}
&&\cm{A}{w(b')}\ar[dd]\ar[rr]\ddrrcompositemap<\omit>{}
&&A\ar[dd]^u
\\
&&&&&\hbox{\kern 20pt et\kern 5pt}
&&
\\
e\ar[rr]_{b'}
&&B'\ar[rr]_{w}
&&B
&&e\ar[rr]_{w(b')}
&&B
&.
\\
&\kern -11pt \mathcal D^{\mathrm{\,g}}_{u',b'}\kern -10pt
&&\kern 3pt\kern -6pt \mathcal D \kern 3pt
&&&&\kern -17pt \mathcal D^{\mathrm{\,g}}_{u,w(b')}\kern -10pt
}
\]
Comme $w$ est une fibration à fibres discrètes, la flèche $\cm{B'}{b'}\toto\cm{B}{w(b')}$, induite par $w$, est un isomorphisme, et comme le carré
\[
\xymatrixcolsep{1.8pc}
\xymatrixrowsep{2.pc}
\xymatrix{
&\cm{A'}{b'}\ar[r]\ar[d]
&\cm{A}{w(b')}\ar[d]
\\
&\cm{B'}{b'}\ar[r]
&\cm{B}{w(b')}
&\hskip -30pt,
}
\]
induit par $\mathcal D$, est cartésien~\cite[dual du lemme 3.2.11]{Ast}, il en est de même de la flèche $\cm{A'}{b'}\toto\cm{A}{w(b')}$, induite par $v$, et on a un isomorphisme
\[
\mathcal D\,\circh \mathcal D^{\mathrm{\,g}}_{u',b'}\,\simeq\, \mathcal D^{\mathrm{\,g}}_{u,w(b')}\smsp.
\]
En vertu de la condition \textbf{CI\,2g}, les carrés $\mathcal D^{\mathrm{\,g}}_{u',b'}$ et $\mathcal D^{\mathrm{\,g}}_{u,w(b')}$ sont $\Q$\nbd-exacts, et la condition~\textbf{\boldmath CS\,$3_{\mathrm h}$} implique alors qu'il en est de même de $\mathcal D$.
\end{proof}

\begin{rem} \label{remcondsuffexfibrdiscr}
Le lemme précédent implique en particulier que si $\Q$ est une classe de flèches satisfaisant aux conditions \textbf{\boldmath CS\,$3_{\mathrm h}$} et \textbf{CI\,2g}, alors pour tout morphisme $u:A\toto B$ de $\Cat$, et tout objet $b$ de $B$, le carré cartésien
\[
\UseTwocells
\xymatrixcolsep{2.5pc}
\xymatrixrowsep{2.4pc}
\xymatrix{
\cm{A}{b}\ar[d]_{}\ar[r]^{}
\drtwocell<\omit>{}
&A\ar[d]^{u}\ar@<-2ex>@{}[d]_(.55){=}
\\
\cm{B}{b}\ar[r]_{}
&B
}
\]
est $\Q$\nbd-exact.
\end{rem}

\begin{lemme} \label{lemmevarder4}
Soit $\Q$ une classe de carrés de $\Cat$ satisfaisant à la condition \emph{\textbf{\boldmath CS\,$1_{\mathrm h}$},} et contenant les carrés de Beck-Chevalley à droite et les carrés cartésiens de la forme
\[
\UseTwocells
\xymatrixcolsep{2.5pc}
\xymatrixrowsep{2.4pc}
\xymatrix{
&\cm{A}{b}\ar[d]_{}\ar[r]^{}
\drtwocell<\omit>{}
&A\ar[d]^{u}\ar@<-2ex>@{}[d]_(.55){=}
\\
&\cm{B}{b}\ar[r]_{}
&B
&\hskip -30pt,
}
\]
pour $u:A\toto B$ flèche de $\Cat$, et $b$ objet de $B$. Alors la classe $\Q$ satisfait aussi à la condition~\emph{\textbf{CI\,2g}.}
\end{lemme}

\begin{proof}
Soient $u:A\toto B$ une flèche de $\Cat$, $b$ un objet de $B$, et considérons le diagramme
\[
\UseTwocells
\xymatrixcolsep{2.5pc}
\xymatrixrowsep{2.4pc}
\xymatrix{
&\cm{A}{b}\ar[d]_{}\ar[r]^{=}
\drtwocell<\omit>{}
&\cm{A}{b}\ar[d]_{}\ar[r]^{}
\drtwocell<\omit>{}
&A\ar[d]^{u}\ar@<-2ex>@{}[d]_(.55){=}
\\
&e\ar[r]_-{(b,\id{b})}
&\cm{B}{b}\ar[r]_{}
&B
&\hskip -30pt.
}
\]
Comme $(b,\id{b})$ est un objet final de $\cm{B}{b}$, la flèche $e\toto\cm{B}{b}$ définie par cet objet admet un adjoint à gauche, et il en est tautologiquement de même pour l'endofoncteur identique de $\cm{A}{b}$. La catégorie but du morphisme de \og changement de base\fg{} correspondant étant la catégorie ponctuelle, ce dernier est forcément un isomorphisme. On en déduit que le carré de gauche est de Beck-Chevalley à droite, et par suite $\Q$\nbd-exact. La condition~\textbf{\boldmath CS\,$1_{\mathrm h}$} implique alors qu'il en est de même du carré composé, ce qui prouve le lemme.
\end{proof}

\section{Carrés exacts homotopiques et dérivateurs}

\begin{paragr} \titparagr{Prédérivateurs} \label{preder}
On rappelle qu'un \emph{prédérivateur} (de domaine $\Cat$) est un $2$\nbd-foncteur (strict) $\D:\op{\Cat}\toto\CAT$ de la $2$\nbd-catégorie des petites catégories vers celle des catégories (non nécessairement petites), contravariant en les $1$\nbd-flèches et en les $2$\nbd-flèches. Si $u:A\toto B$ est un morphisme de $\Cat$, le foncteur $\D(u):\D(B)\toto\D(A)$ est noté le plus souvent $u_{\D}^*\,$, ou même simplement $u^*$, quand aucune confusion n'en résulte. De même, si $\alpha$ est une transformation naturelle dans $\Cat$, alors $\D(\alpha)$ est noté $\alpha_\D^*$ ou $\alpha^*$. 
\[
\UseTwocells
{
\xymatrixcolsep{2.5pc}
\xymatrix{
A\rrtwocell<6>^u_v{\,\alpha}&&B\\
}
}
\quad\mapstoto\quad
\xymatrix{
\D(A)&&\D(B)\lltwocell<6>^{v^*}_{u^*}{\alpha^*\ }\\
}\quad
\]
Un prédérivateur $\D$ transforme un carré 
\[
\mathcal{D}\ =\quad
\raise 23pt
\vbox{
\UseTwocells
\xymatrixcolsep{2.5pc}
\xymatrixrowsep{2.4pc}
\xymatrix{
A'\ar[d]_{u'}\ar[r]^{v}
\drtwocell<\omit>{\alpha}
&A\ar[d]^{u}
\\
B'\ar[r]_{w}
&B
}
}\qquad
\]
de $\Cat$ en un carré
\[
\D(\mathcal{D})\ =\quad
\raise 23pt
\vbox{
\UseTwocells
\xymatrixcolsep{2.5pc}
\xymatrixrowsep{2.4pc}
\xymatrix{
\D(B)\ar[d]_{u^*}\ar[r]^{w^*}
\drtwocell<\omit>{\alpha^*}
&\D(B')\ar[d]^{{u'}^*}
\\
\D(A)\ar[r]_{v^*}
&\D(A')
}
}\qquad
\]
de $\CAT$. Si $\mathcal{D}$ et $\mathcal{D}'$ sont deux carrés de $\Cat$ composables horizontalement (resp. verticalement), alors les carrés $\D(\mathcal{D}')$ et $\D(\mathcal{D})$ de $\CAT$ le sont aussi, et on a les égalités
\begin{equation}
\D(\mathcal D\circh\mathcal D')=\D(\mathcal D')\circh\D(\mathcal D)\qquad
\hbox{(resp.}\quad
\D(\mathcal D\circv\mathcal D')=\D(\mathcal D')\circv\D(\mathcal D)
\ \ ).
\end{equation}
Si $\mathcal{D}$ est un carré de Beck-Chevalley à gauche (resp. à droite) de $\Cat$, il en est de même, par 2-fonctorialité, du carré $\D(\mathcal{D})$ de $\CAT$.
\end{paragr}

\begin{paragr} \titparagr{Dérivateurs} \label{defder}
On rappelle qu'un \emph{dérivateur}~\cite{Der},~\cite{Mal1} est un prédérivateur $\D$ satisfaisant les axiomes suivants.
\smallskip

\noindent
\textbf{Der\,1} \ \textbf{(\emph{Normalisation.})} Pour toute famille finie $(A_i)_{i\in I}$ de petites catégories le foncteur canonique
\[
\D\bigl({\textstyle\coprod\limits_i}A_i\bigr)\toto{\textstyle\prod\limits_i}\D(A_i)
\]
est une équivalence de catégories.
\smallbreak

\noindent
\textbf{Der\,2} \ \textbf{(\emph{Conservativité.})} Pour toute petite catégorie $A$, et toute flèche $\varphi:X\toto Y$ de $\D(A)$, si pour tout objet $a$ de $A$, la flèche $a^*(\varphi):a^*(X)\toto a^*(Y)$ (où $a:e\toto A$ désigne aussi le foncteur de la catégorie ponctuelle $e$ vers $A$, défini par $a$) est un isomorphisme de $\D(e)$, alors $\varphi$ est un isomorphisme de $\D(A)$ (autrement dit, la famille des foncteurs $a^*:\D(A)\toto\D(e)$, $a\in\Ob\,A$, est conservative).
\smallbreak

\noindent
\textbf{Der\,3g} \ \textbf{(\emph{Existence d'images directes cohomologiques.})} Pour toute flèche \hbox{$u:A\toto B$} de $\Cat$, le foncteur \emph{image inverse} $u^*:\D(B)\toto\D(A)$ admet un adjoint à droite $u^{}_*=u^{\D}_*:\D(A)\toto\D(B)$, foncteur \emph{image directe cohomologique}.
\smallbreak

\noindent
\textbf{Der\,3d} \ \textbf{(\emph{Existence d'images directes homologiques.})} Pour toute flèche \hbox{$u:A\toto B$} de $\Cat$, le foncteur image inverse $u^*:\D(B)\toto\D(A)$ admet un adjoint à gauche $u^{}_!=u^{\D}_!:\D(A)\toto\D(B)$, foncteur \emph{image directe homologique}.
\smallbreak

\noindent
\textbf{Der\,4g} \ \textbf{(\emph{Calcul des fibres des images directes cohomologiques.})} Pour toute flèche \hbox{$u:A\toto B$} de $\Cat$, et tout objet $b$ de $B$, l'image par $\D$ du carré comma
\[
\mathcal{D}_{u,b}^{\mathrm{\,g}}\ =\quad
\raise 23pt
\vbox{
\UseTwocells
\xymatrixcolsep{2.5pc}
\xymatrixrowsep{2.4pc}
\xymatrix{
\cm{A}{b}\ar[d]_{p}\ar[r]^{j}
\drtwocell<\omit>{\alpha}
&A\ar[d]^{u}
\\
e\ar[r]_{b}
&B
}
}\qquad
\]
est un carré de Beck Chevalley à gauche (autrement dit, le morphisme canonique $b^*u^{}_*\toto p^{}_*j^*$ est un isomorphisme).
\smallbreak

\noindent
\textbf{Der\,4d} \ \textbf{(\emph{Calcul des fibres des images directes homologiques.})} Pour toute flèche \hbox{$u:A\toto B$} de $\Cat$, et tout objet $b$ de $B$, l'image par $\D$ du carré comma
\[
\mathcal{D}_{u,b}^{\mathrm{\,d}}\ =\quad
\raise 23pt
\vbox{
\UseTwocells
\xymatrixcolsep{2.5pc}
\xymatrixrowsep{2.4pc}
\xymatrix{
\mc{b}{A}\ar[d]_{k}\ar[r]^{q}
\drtwocell<\omit>{\beta}
&e\ar[d]^{b}
\\
A\ar[r]_{u}
&B
}
}\qquad
\]
est un carré de Beck Chevalley à droite (autrement dit, le morphisme canonique $b^*u^{}_!\otot q^{}_!k^*$ est un isomorphisme).
\end{paragr}

\begin{paragr} \titparagr{Terminologie plus précise} \label{termdefder}
Soit $\D$ un prédérivateur. On dit que $\D$ est \emph{conservatif} s'il satisfait à l'axiome \textbf{Der\,2}, on dit qu'il \emph{admet des images directes cohomologiques} (resp.~\emph{homologiques}) s'il satisfait à l'axiome \textbf{Der\,3g} (resp.~\textbf{Der\,3d}). De façon plus précise, 
on dit qu'un morphisme $u:A\toto B$ de $\Cat$ \emph{admet une image directe cohomologique} (resp. \emph{homologique}) \emph{pour} $\D$ si le foncteur \hbox{$u^*:\D(B)\toto\D(A)$} admet un adjoint à droite (resp. à gauche) $u^{}_*:\D(A)\toto\D(B)$ (resp.\break \hbox{$u^{}_!:\D(A)\toto\D(B)$}).
On dit que le prédérivateur $\D$ est \emph{complet} (resp.~\emph{cocomplet}) s'il satisfait aux axiomes \textbf{Der\,3g} et \textbf{Der\,4g} (resp.~\textbf{Der\,3d} et \textbf{Der\,4d}). Un \emph{pseudo-dérivateur  faible à gauche} (resp.~\emph{à droite}) est un prédérivateur conservatif et complet (resp. cocomplet). S'il satisfait de plus à l'axiome \textbf{Der\,1}, on dit qu'il est un \emph{dérivateur  faible à gauche} (resp.~\emph{à droite}). Un dérivateur est donc un dérivateur faible à gauche \emph{et} à droite. Un \emph{pseudo-dérivateur} est un pseudo-dérivateur faible à gauche et à droite, autrement dit, un prédérivateur satisfaisant à \emph{tous} les axiomes d'un dérivateur \emph{sauf} \textbf{Der\,1}. La raison d'être de l'adjectif \og faible\fg{} sera expliquée ultérieurement~(\ref{derpasfaible}), où les variantes \og non faibles\fg{} seront introduites. Dans cet article, l'axiome de normalisation \textbf{Der\,1} ne joue aucun rôle. On s'intéressera donc surtout aux pseudo-dérivateurs, ainsi qu'à leurs variantes à gauche ou à droite.
\end{paragr}

\begin{paragr} \titparagr{Exemples de dérivateurs} \label{exder}
Il existe de nombreux exemples de dérivateurs.

\begin{subparagr} \titsubparagr{Le dérivateur des préfaisceaux} \label{derpref}
Soit $\C$ une catégorie. On définit un prédérivateur $\D_{\C}$, associant à toute petite catégorie $A$ la catégorie 
\[
\D_{\C}(A)=\Homint(\op{A},\C)
\] 
des préfaisceaux sur $A$ à valeurs dans $\C$, à tout foncteur $u:A\toto B$ entre petites catégories le foncteur
\[
u^*:\D_{\C}(B)\toto\D_{\C}(A)\smsp,\qquad G\mapstoto G\circ\op{u}
\]
et à toute transformation naturelle $\alpha:u\toto v$ entre foncteurs de $A$ vers $B$ dans $\Cat$, le morphisme de foncteurs $\alpha^*:v^*\toto u^*$, défini par
\[
\textstyle\alpha^*_{G,a}=G(\alpha_a):v^*(G)(a)=G(v(a))\toto G(u(a))=u^*(G)(a)\,,\ G:\op{B}\toto\C\,,\ a\in\Ob\,A\,.
\]
Ce prédérivateur satisfait toujours, sans aucune hypothèse sur $\C$, les axiomes \textbf{Der\,1} et \textbf{Der\,2}. Si la catégorie $\C$ est complète, il satisfait aussi aux axiomes \textbf{Der\,3g} et \textbf{Der\,4g}, le premier exprimant l'existence des extensions de Kan à droite, et le second leur calcul habituel (le foncteur $p^{}_*$, dans les notations de cet axiome, étant alors simplement le foncteur limite projective \smash{$\varprojlim_{\op{(\cm{A}{b})}}$}). Dualement, si la catégorie $\C$ est cocomplète, les axiomes \textbf{Der\,3d} et \textbf{Der\,4d} sont satisfaits. Ainsi, si la catégorie $\C$ est à la fois complète et cocomplète, $\D_{\C}$ est un dérivateur.
\end{subparagr}

\begin{subparagr} \titsubparagr{Le dérivateur associé à une catégorie de modèles} \label{dercatmod}
Soient $\C$ une catégorie et $W$ une classe de flèches de $\C$. On définit un prédérivateur $\D_{\C,W}$ en associant à toute petite catégorie $A$ la catégorie $\D_{\C,W}(A)$, obtenue de la catégorie des préfaisceaux sur $A$ à valeurs dans $\C$ en inversant formellement les morphismes de préfaisceaux qui sont dans $W$ argument par argument:
\[
\begin{aligned}
&\D_{\C,W}(A)=W^{-1}_A\D_{\C}(A)=W^{-1}_A\Homint(\op{A},\C)\smsp,\cr
\noalign{\vskip 3pt}
&W_A=\{\varphi\in\Fl\,\Homint(\op{A},\C)\mid\forall a\in\Ob\,A,\,\varphi^{}_a\in W\}\smsp,
\end{aligned}
\] 
les foncteurs et morphismes de foncteurs $u^*_{\D_{\C,W}}$ et $\alpha^*_{\D_{\C,W}}$ étant déduits des $u^*_{\D_{\C}}$ et $\alpha^*_{\D_{\C}}$ à l'aide de la propriété universelle de la localisation. Si la catégorie $\C$ est complète et cocomplète, et s'il existe une structure de catégorie de modèles de Quillen sur $\C$~\cite{Qu0}, avec $W$ comme classe d'équivalences faibles, alors le prédérivateur $\D_{\C,W}$ est un dérivateur~\cite{CiDer}. En fait, il suffit des conditions beaucoup plus faibles~\cite{CiCatDer}.
\end{subparagr}
\end{paragr}

\begin{paragr} \titparagr{Localisateur fondamental associé à un dérivateur} \label{locfondder}
Soit $\D$ un prédérivateur. On dit qu'une flèche $u:A\toto B$ de $\Cat$ est une $\D$\nbd-\emph{équivalence} si le foncteur $u^*:\D(B)\toto\D(A)$ induit un foncteur pleinement fidèle sur la sous-catégorie de $\D(B)$ formée des objets de la forme $q^*(X)$ avec $X$ objet de $\D(e)$, $q$ étant le foncteur $B\toto e$ de $B$ vers la catégorie ponctuelle. En d'autres termes, si l'on pose $p=qu:A\toto e$,
\[
\xymatrixcolsep{.7pc}
\xymatrix{
A\ar[rr]^{u}\ar[dr]_{p}
&&B\ar[dl]^{q}
\\
&e
}
\]
pour que la flèche $u$ soit une $\D$\nbd-équivalence, il faut et il suffit que pour tout couple d'objets $X,Y$ de $\D(e)$, l'application 
\[
\Hom^{}_{\D(B)}(q^*(X),q^*(Y))\toto\Hom^{}_{\D(A)}(p^*(X),p^*(Y))\smsp,
\]
induite par $u^*$, soit bijective. On en déduit que si le prédérivateur $\D$ admet des images directes cohomologiques (resp. homologiques), alors la flèche $u$ est une $\D$\nbd-équivalence si et seulement si le morphisme canonique de foncteurs
\[
q^{}_*q^*\toto p^{}_*p^*\qquad \hbox{(resp. \ }p^{}_!p^*\toto q^{}_!q^*\  )
\]
est un isomorphisme. On note $\W_{\D}$ la partie de $\Fl(\Cat)$ formée des $\D$\nbd-équivalences. On démontre que si le prédérivateur $\D$ est conservatif et complet (ou cocomplet), alors $\W_{\D}$ est un localisateur fondamental~\cite{Der},~\cite{Mal1}. Ainsi, on dispose alors de toutes les notions associées à un localisateur fondamental. On dira que la petite catégorie $A$ est $\D$\nbd-\emph{asphérique} si elle est $\W_{\D}$\nbd-asphérique, autrement dit, si le foncteur \hbox{$p^*:\D(e)\toto\D(A)$} est pleinement fidèle. De même, on dira que la flèche \hbox{$u:A\toto B$} est $\D$\nbd-\emph{asphérique}, $\D$\nbd-\emph{coasphérique}, une $\D$\nbd-\emph{équivalence universelle}, $\D$\nbd-\emph{lisse}, ou $\D$\nbd-\emph{propre}, si elle est $\W_{\D}$\nbd-asphérique, $\W_{\D}$\nbd-coasphérique, une $\W_{\D}$\nbd-équivalence universelle, $\W_{\D}$\nbd-lisse, ou $\W_{\D}$\nbd-propre respectivement. Si
\[
\xymatrixcolsep{.7pc}
\xymatrix{
A\ar[rr]^{u}\ar[dr]_{v}
&&B\ar[dl]^{w}
\\
&C
}
\]
est un triangle commutatif dans $\Cat$, on dira que $u$ est une $\D$\nbd-\emph{équivalence localement}, ou \emph{colocalement}, sur $C$ si elle est respectivement une $\W_\D$\nbd-équivalence localement, ou colocalement, sur $C$. Enfin, on dira qu'un carré de $\Cat$ est $\D$\nbd-\emph{exact} (resp.~$\D$\nbd-\emph{exact faible}) s'il est $\W_\D$\nbd-exact (resp.~$\W_\D$\nbd-exact faible).
\end{paragr}

\begin{ex} \label{exlocfondder}
Soient $\C$ une catégorie, et $\D=\D_{\C}$ le prédérivateur des préfaisceaux à valeurs dans $\C$ (\cf~\ref{derpref}). Alors on a (dans les notations des exemples~\ref{defWn} et~\ref{defWgr})
\[
\W_{\D}=
\left\{
\begin{aligned}
&\Wzer\smsp,\hskip 15pt\hbox{si $\C$ n'est pas une catégorie associée à un ensemble préordonné;}\cr
\noalign{\vskip 3pt}
&\Wgr\smsp,\hskip 10.9pt\hbox{si $\C$ est une catégorie associée à un ensemble préordonné non vide,}\cr
\noalign{\vskip -3pt}
&\hskip 57pt \hbox{et n'est pas équivalente à la catégorie ponctuelle;}\cr
\noalign{\vskip 3pt}
&\Wtr\smsp,\hskip 12pt\hbox{si $\C$ est vide ou équivalente à la catégorie ponctuelle.}
\end{aligned}
\right.
\]
En effet, on observe que pour toute petite catégorie $A$, et tout couple $X$ et $Y$ d'objets de $\C$, l'ensemble des morphismes du préfaisceau constant sur $A$ de valeur $X$ vers celui de valeur $Y$ est en bijection avec l'ensemble $\Hom^{}_{\C}(X,Y)^{\pi_{0}(A)}$, où $\pi_{0}(A)$ désigne l'ensemble des composantes connexes de la catégorie $A$. Par définition, dire qu'une flèche $u:A\toto B$ de $\Cat$ est une $\D$-équivalence signifie donc que pour tout couple $X$ et $Y$ d'objets de $\C$, l'application 
\[
\Hom^{}_{\C}(X,Y)^{\pi_{0}(B)}\toto\Hom^{}_{\C}(X,Y)^{\pi_{0}(A)}\smsp,
\]
définie en précomposant avec l'application $\pi_0(u):\pi_0(A)\toto\pi_0(B)$, est bijective. S'il existe des objets $X$, $Y$ de $\C$ tels que $\Hom^{}_{\C}(X,Y)$ ait au moins deux éléments, on vérifie facilement que cela équivaut à la bijectivité de l'application $\pi_0(u)$ elle-même. Si pour tout couple $X$ et $Y$ d'objets de $\C$, l'ensemble $\Hom^{}_{\C}(X,Y)$ a au plus un élément, mais il existe un couple pour lequel cet ensemble est vide, cela signifie simplement que les ensembles $\pi_{0}(A)$ et $\pi_{0}(B)$ sont tout deux vides ou tout deux non vides.
\end{ex}

\begin{ex}
Soient $\W$ un localisateur fondamental, et $\D=\D_{\Cat,\W}$ le pré\-dé\-ri\-va\-teur associé (\cf~\ref{dercatmod}). On démontre qu'on a alors $\W_{\D}=\W$~\cite[proposition~3.1.10,~(\emph{a})]{Ast}
\end{ex}

\begin{paragr} \titparagr{Carrés satisfaisant à la propriété de changement de base} \label{proprchangbase}
Soit $\D$ un prédérivateur. On dit qu'un carré 
\[
\mathcal{D}\ =\quad
\raise 23pt
\vbox{
\UseTwocells
\xymatrixcolsep{2.5pc}
\xymatrixrowsep{2.4pc}
\xymatrix{
A'\ar[d]_{u'}\ar[r]^{v}
\drtwocell<\omit>{\alpha}
&A\ar[d]^{u}
\\
B'\ar[r]_{w}
&B
}
}\qquad
\]
de $\Cat$ satisfait à la propriété de \emph{changement de base cohomologique} (resp. \emph{homologique}) \emph{pour} $\D$ si le carré $\D(\mathcal D)$ (\cf~\ref{preder}) est un carré de Beck-Chevalley à gauche (resp. à droite) (\cf~\ref{BeckChev}), autrement dit, si les morphismes $u$ et $u'$ (resp. $v$ et $w$) admettent des images directes cohomologiques (resp. homologiques), et si le \og morphisme de changement de base\fg
\[
w^*u^{}_*\toto u'_*v^*\qquad\hbox{(resp. \ }v^{}_!u'^*\toto u^*w^{}_! \ )
\]
est un isomorphisme. Quand aucune ambiguïté n'en résulte, on omettra la mention explicite du prédérivateur $\D$. Si les morphismes $u$, $u'$ admettent des images directes cohomologiques \emph{et} $v$, $w$ des images directes homologiques, alors $\mathcal{D}$ satisfait à la propriété de changement de base cohomologique si et seulement si il satisfait à la propriété de changement de base homologique (\cf~\ref{BeckChev}). On dira parfois dans ce cas, plus simplement, qu'il satisfait à la \emph{propriété de changement de base}. On a la proposition soritale suivante.
\end{paragr}

\begin{prop} \label{sorchangbase}
Soit $\D$ un prédérivateur.
\smallskip

\emph{i)} Tout carré de Beck-Chevalley à gauche \emph{(}resp. à droite\emph{)} satisfait à la propriété de changement de base cohomologique \emph{(}resp. homologique\emph{).}
\smallbreak

\emph{ii)} La classe des carrés de $\Cat$ satisfaisant à la propriété de changement de base cohomologique \emph{(}resp. homologique\emph{)} est stable par composition horizontale \emph{et} verticale.
\smallbreak

\emph{iii)} Si le prédérivateur $\D$ est complet \emph{(}resp. cocomplet\emph{),} pour toute flèche $u:A\toto B$ de $\Cat$, et tout objet $b$ de $B$, le carré comma $\mathcal{D}_{u,b}^{\mathrm{\,g}}$ \emph{(}resp. $\mathcal{D}_{u,b}^{\mathrm{\,d}}$\emph{)} satisfait à la propriété de changement de base cohomologique \emph{(}resp. homologique\emph{)}.
\end{prop}

\begin{proof}
La première assertion résulte simplement de la $2$\nbd-fonctorialité de~$\D$, la deuxième de la stabilité de la classe des carrés de Beck-Chevalley à gauche (resp. à droite) par composition horizontale et verticale (\cf~\ref{classcarBCg} (resp.~\ref{classcarBCd})), et la troisième n'est qu'une reformulation de l'axiome \textbf{Der\,4g} (resp. \textbf{Der\,4d}).
\end{proof}

\begin{cor} \label{corsorchangbase}
Soit $\D$ un prédérivateur admettant des images directes cohomologiques. Alors tout carré de Beck-Chevalley à droite satisfait à la propriété de changement de base cohomologique.
\end{cor}

\begin{proof}
En vertu de l'assertion (\emph{i}) de la proposition précédente, l'image par $\D$ d'un tel carré est un carré de Beck-Chevalley à droite, et comme par hypothèse les foncteurs figurant dans ce carré admettent des adjoints à droite, ce carré est aussi un carré de Beck-Chevalley à gauche (\cf~\ref{BeckChev}), ce qui prouve l'assertion.
\end{proof}

\begin{lemme} \label{Der2gen}
Soient $\D$ un prédérivateur conservatif, et $w_j:B_j\toto B$, $j\in J$, une famille de flèches de $\Cat$, de même but $B$. Si \smash{$\Ob\,B=\bigcup\limits_{j\in J}w_j(\Ob\,B_j)$}, alors la famille des foncteurs $w_j^*$ est conservative.
\end{lemme}

\begin{proof} Soit $\varphi:X\toto Y$ une flèche de $\D(B)$ telle que pour tout $j\in J$, $w^*_j(\varphi)$ soit un isomorphisme. Pour tout objet $b$ de $B$, il existe $j\in J$, et un objet $b_j$ de $B_j$ tel que $b=w_j(b_j)$. Si on note aussi $b:e\toto B$ et $b_j:e\toto B_j$ les foncteurs définis par les objets $b$ et $b_j$ respectivement, on a $b^*(\varphi)=b_j^*w^*_j(\varphi)$, et par suite, $b^*(\varphi)$ est un isomorphisme. Comme le prédérivateur $\D$ est conservatif, on en déduit que $\varphi$ est un isomorphisme.
\end{proof}

\begin{prop} \label{descchangbase}
Soit $\D$ un prédérivateur conservatif admettant des images directes cohomologiques. Alors la classe des carrés satisfaisant à la propriété de changement de base cohomologique vérifie la condition de descente horizontale.
\end{prop}

\begin{proof}
Soient $J$ un ensemble, et 
\[
\xymatrixrowsep{.7pc}
\xymatrixcolsep{1.3pc}
\UseAllTwocells
\xymatrix{
&A'\ddrrcompositemap<\omit>{\alpha}
  \ar[rr]^{v}
  \ar[dd]_{u'}
&&A\ar[dd]^u
&&&A_j\ddrrcompositemap<\omit>{\ \alpha_j}
  \ar[rr]^{v_j}
  \ar[dd]_{u_j}
&&A'\ar[dd]^{u'}
\\
{\mathcal D}\ =
&&&&\hbox{et}
&{\mathcal D_j}\ =
&&&&,\ j\in J\ ,
\\
&B'\ar[rr]_w
&&B
&&&B_j\ar[rr]_{w_j}
&&B'
}
\]
des carrés dans $\Cat$ tels que $\Ob\,B'=\bigcup\limits_{j\in J}w_j(\Ob\,B_j)$, et tels que pour tout élément $j$ de $J$, les carrés $\mathcal D_j$ et $\mathcal D\circh\mathcal D_j$ satisfassent à la propriété de changement de base cohomologique. Notons 
\[
c^{}_{\mathcal D}:w^*u^{}_*\toto u'_*v^*\qquad\hbox{et}\qquad c^{}_{\mathcal D_j}:w^*_ju'_*\toto u^{}_j{}^{}_*v_j^*
\]
les morphismes de changement de base relatifs aux carrés $\mathcal D$ et $\mathcal D_j$. On vérifie aussitôt que le morphisme de changement de base relatif au carré composé $\mathcal D\circh\mathcal D_j$ est défini (pour un choix convenable des morphismes d'adjonction) par la formule
\[
c^{}_{\mathcal D\circh\mathcal D_j}=(c^{}_{\mathcal D_j}\star v^*)\circ(w_j^*\star c^{}_{\mathcal D}):(ww^{}_j)^*u^{}_*=w_j^*w^*u^{}_*\toto u^{}_j{}^{}_*v_j^*v^*=u^{}_j{}^{}_*(vv^{}_j)^*\smsp.
\]
Or par hypothèse, pour tout $j\in J$, les morphismes $c^{}_{\mathcal D_j}$ et $c^{}_{\mathcal D\circh\mathcal D_j}$ sont des isomorphismes, donc $w_j^*\star c^{}_{\mathcal D}$ aussi. Le lemme précédent implique alors que $c^{}_{\mathcal D}$ est un isomorphisme.
\end{proof}

\begin{cor} \label{desclocchangbase}
Soit $\D$ un prédérivateur conservatif et complet. Alors la classe des carrés satisfaisant à la propriété de changement de base cohomologique vérifie la condition de descente locale.
\end{cor}

\begin{proof}
Le corollaire résulte aussitôt de la proposition précédente, et de la proposition~\ref{sorchangbase}, (\emph{iii}) (\cf~\ref{implev}).
\end{proof}

\begin{thm} \label{derfgchangbase}
Soit $\D$ un prédérivateur conservatif et complet.
\smallskip

\emph{i)} Tout carré comma satisfait à la propriété de changement de base cohomologique.
\smallbreak

\emph{ii)} Tout carré cartésien de la forme
\[
\UseTwocells
\xymatrixcolsep{2.5pc}
\xymatrixrowsep{2.4pc}
\xymatrix{
&A'\ar[d]_{u'}\ar[r]^{v}
\drtwocell<\omit>{}
&A\ar[d]^{u}\ar@<-2ex>@{}[d]_(.55){=}
\\
&B'\ar[r]_{w}
&B
&\hskip -30pt,
}
\]
avec $u$ précofibration ou $w$ fibration discrète, satisfait à la propriété de changement de base cohomologique.
\end{thm}

\begin{proof}
Le théorème résulte des lemmes~\ref{condsuffexcomma}, \ref{condsuffexpropre} et~\ref{condsuffexfibrdiscr}, en observant que les hypothèses de ces lemmes sont satisfaites, grâce aux propositions~\ref{sorchangbase}, (\emph{ii}), (\emph{iii}) et~\ref{descchangbase}, et aux corollaires~\ref{corsorchangbase} et~\ref{desclocchangbase}, en tenant compte de l'exemple~\ref{classcarBCd}.
\end{proof}

\begin{rem}
La première partie du théorème précédent implique aussitôt que si un prédérivateur conservatif et complet admet des images directes homologiques, alors il est forcément cocomplet. En particulier, dans la définition d'un dérivateur, l'axiome \textbf{Der\,4d} est superflu, étant conséquence des autres axiomes. Dualement, on peut omettre l'axiome \textbf{Der\,4g} (mais pas les deux à la fois!).
\end{rem}

\begin{prop}
Soit $\D$ un prédérivateur conservatif admettant des images directes cohomologiques. Les conditions suivantes sont équivalentes:
\begin{GMitemize}
\item[\rm a)] $\D$ est complet;
\item[\rm b)] pour toute flèche $u:A\toto B$, et tout objet $b$ de $B$, le carré comma
\[
\mathcal{D}_{u,b}^{\mathrm{\,g}}\ =\quad
\raise 23pt
\vbox{
\UseTwocells
\xymatrixcolsep{2.5pc}
\xymatrixrowsep{2.4pc}
\xymatrix{
\cm{A}{b}\ar[d]_{}\ar[r]^{}
\drtwocell<\omit>{}
&A\ar[d]^{u}
\\
e\ar[r]_{b}
&B
}
}\qquad
\]
satisfait à la propriété de changement de base cohomologique;
\item[\rm c)] pour toute flèche $u:A\toto B$, et tout objet $b$ de $B$, le carré cartésien
\[
\UseTwocells
\xymatrixcolsep{2.5pc}
\xymatrixrowsep{2.4pc}
\xymatrix{
&\cm{A}{b}\ar[d]_{}\ar[r]^{}
\drtwocell<\omit>{}
&A\ar[d]^{u}\ar@<-2ex>@{}[d]_(.55){=}
\\
&\cm{B}{b}\ar[r]_{}
&B
&\hskip -30pt,
}
\]
satisfait à la propriété de changement de base cohomologique.
\end{GMitemize}
\end{prop}

\begin{proof}
L'équivalence des conditions (\emph{a}) et (\emph{b}) vient d'une simple reformulation des définitions. L'implication (\emph{b}) $\Rightarrow$ (\emph{c}) résulte de la remarque~\ref{remcondsuffexfibrdiscr}, de la proposition~\ref{sorchangbase}, (\emph{iii}), et du corollaire~\ref{desclocchangbase}. L'implication (\emph{c}) $\Rightarrow$ (\emph{b}) est conséquence du lemme~\ref{lemmevarder4}, dont les hypothèses sont satisfaites, grâce à la proposition~\ref{sorchangbase}, (\emph{ii}), et au corollaire~\ref{corsorchangbase}
\end{proof}

\begin{rem}
La proposition précédente implique que dans la définition d'un dérivateur, l'axiome \textbf{Der\,4g} peut être remplacé par la condition (\emph{c}) ci-dessus, et dualement pour l'axiome \textbf{Der\,4d}.
\end{rem}

\begin{prop} \label{carcatDasph}
Soient $\D$ un prédérivateur conservatif et complet, et $A$ une petite catégorie. Considérons le carré
\[
\hskip 15pt
\mathcal{D}\ =\quad
\raise 23pt
\vbox{
\UseTwocells
\xymatrixcolsep{2.5pc}
\xymatrixrowsep{2.4pc}
\xymatrix{
A\ar[d]_{p}\ar[r]^{p}
\drtwocell<\omit>{}
&e\ar[d]^{\vrule height 3pt depth 2pt width .3pt\kern 1.2pt\vrule height 3pt depth 2pt width .3pt}\ar@<-2ex>@{}[d]_(.55){=}\ar@{}[d]
\\
e\ar[r]_{=}
&e
&\hskip -30pt,
}
}
\]
où $e$ désigne la catégorie ponctuelle. Les conditions suivantes sont équivalentes:
\begin{GMitemize}
\item[\rm a)] $A$ est $\D$-asphérique;
\item[\rm b)] le carré $\mathcal D$ est $\D$-exact;
\item[\rm c)] le carré $\mathcal D$ satisfait à la propriété de changement de base cohomologique;
\item[\rm d)] le morphisme d'adjonction $\id{\D(e)}\toto p^{}_*p^*$ est un isomorphisme.
\end{GMitemize}
\end{prop}

\begin{proof}
L'équivalence des conditions (\emph{a}) et (\emph{b}) est immédiate (\cf~\ref{extrcarex}). De même, l'équivalence des conditions (\emph{c}) et (\emph{d}) est tautologique, puisque le morphisme d'adjonction $\id{\D(e)}\simeq(\id{e})^*(\id{e})_*\toto p^{}_*p^*$ n'est autre que le morphisme de changement de base relatif au carré $\mathcal D$. Enfin, par définition (\cf~\ref{locfondder}), la catégorie $A$ est $\D$\nbd-asphérique si et seulement si le foncteur $p^*$ est pleinement fidèle, ce qui prouve l'équivalence des conditions (\emph{a}) et (\emph{d}).
\end{proof}

\begin{prop} \label{carfonctDasph}
Soient $\D$ un prédérivateur conservatif et complet, et $u:A\toto B$ une flèche de $\Cat$. Considérons le carré
\[
\hskip 15pt
\mathcal{D}\ =\quad
\raise 23pt
\vbox{
\UseTwocells
\xymatrixcolsep{2.5pc}
\xymatrixrowsep{2.4pc}
\xymatrix{
A\ar[d]_{u}\ar[r]^{p}
\drtwocell<\omit>{}
&e\ar[d]^{\vrule height 3pt depth 2pt width .3pt\kern 1.2pt\vrule height 3pt depth 2pt width .3pt}\ar@<-2ex>@{}[d]_(.55){=}\ar@{}[d]
\\
B\ar[r]_{q}
&e
&\hskip -30pt,
}
}
\]
où $e$ désigne la catégorie ponctuelle. Les conditions suivantes sont équivalentes:
\begin{GMitemize}
\item[\rm a)] $u$ est $\D$-asphérique;
\item[\rm b)] le carré $\mathcal D$ est $\D$-exact;
\item[\rm c)] le carré $\mathcal D$ satisfait à la propriété de changement de base cohomologique;
\item[\rm d)] le morphisme canonique $q^*\toto u^{}_*p^*=u^{}_*u^*q^*$, défini par le morphisme d'adjonction, est un isomorphisme.
\end{GMitemize}
\end{prop}

\begin{proof}
L'équivalence des conditions (\emph{a}) et (\emph{b}) est immédiate (\cf~\ref{extrcarex}). De même, l'équivalence des conditions (\emph{c}) et (\emph{d}) est tautologique, puisque le morphisme canonique $q^*\simeq q^*{\id{e}}_*\toto u^{}_*p^*$ n'est autre que le morphisme de changement de base relatif au carré $\mathcal D$. Montrons l'équivalence des conditions (\emph{a}) et (\emph{c}). 
\gmremarque{
En vertu de l'axiome \textbf{Der\,2}, pour que le morphisme $q^*\toto u^{}_*p^*$ soit un isomorphisme, il faut et il suffit que pour tout objet $b$ de $B$, la flèche $b^*q^*\toto b^*u^{}_*p^*$ (où $b:e\toto B$ désigne le foncteur défini par l'objet $b$ de $B$) le soit. Or, $b^*q^*=(qb)^*={\id{e}}^*=\id{\D(e)}$, et en vertu de l'axiome \textbf{Der\,4g} on a un isomorphisme
\[
b^*u^{}_*\Toto{1.3}{\sim}r^{}_*j^*
\hskip 30pt
\raise 23pt
\vbox{
\UseTwocells
\xymatrixcolsep{2.5pc}
\xymatrixrowsep{2.4pc}
\xymatrix{
\cm{A}{b}\ar[d]_{r}\ar[r]^{j}
\drtwocell<\omit>{\alpha}
&A\ar[d]^{u}
\\
e\ar[r]_{b}
&B
}
}\qquad
\]
}
Soit $b$ un objet de $B$, et considérons le diagramme
\[
\xymatrixrowsep{.9pc}
\xymatrixcolsep{.9pc}
\UseAllTwocells
\xymatrix{
\cm{A}{b}\ddrrcompositemap<\omit>{}
  \ar[rr]^{}
  \ar[dd]
&&A\ar[dd]^{u}
  \ddrrcompositemap<\omit>{}
  \ar[rr]^{p}
&&e\ar[dd]^{\vrule height 3pt depth 2pt width .3pt\kern 1.2pt\vrule height 3pt depth 2pt width .3pt}
  \ar[dd]^(.55){\hskip -24pt =}
\\
\\
e\ar[rr]_{b}
&&B\ar[rr]_{q}
&&e
&.
\\
&\kern -11pt \mathcal D^{\mathrm{\,g}}_{u,b}\kern -10pt
&&\kern 3pt\kern -6pt \mathcal D \kern 3pt
}
\]
Il résulte de la proposition~\ref{sorchangbase}, (\emph{ii}), (\emph{iii}), et du corollaire~\ref{desclocchangbase}, que le carré $\mathcal D$ satisfait à la propriété de changement de base cohomologique si et seulement si pour tout objet $b$ de $B$, le carré composé $\mathcal D\circh\mathcal D^{\mathrm{\,g}}_{u,b}$ vérifie cette propriété, autrement dit en vertu de la proposition précédente, si et seulement si la catégorie $\cm{A}{b}$ est $\D$\nbd-asphérique, ce qui prouve l'assertion.
\end{proof}

\begin{rem}
Soit $\D$ un prédérivateur conservatif et complet. Dans les notation de la section précédente, on a montré que la classe des carrés satisfaisant à la propriété de changement de base cohomologique pour $\D$ vérifie les conditions de stabilité \textbf{\boldmath CS\,$1_{\mathrm h}$},
\textbf{\boldmath CS\,$1_{\mathrm v}$}, \textbf{\boldmath CS\,$2_{\mathrm h}$}, \textbf{\boldmath CS\,$3_{\mathrm h}$} (\cf~propositions~\ref{sorchangbase}, (\emph{ii}), et \ref{descchangbase}, et corollaire~\ref{desclocchangbase}) et les conditions \og d'initialisation\fg{} \textbf{CI\,1g} (relativement au localisateur fondamental des $\D$\nbd-équivalences), \textbf{\boldmath CI\,$1'$d}, \textbf{CI\,2g} (\cf~propositions~\ref{sorchangbase}, (\emph{iii}), et \ref{carfonctDasph}, et corollaire~\ref{corsorchangbase}), et donc aussi, à plus forte raison, les conditions~\textbf{\boldmath CI\,$1_0$} (relativement au localisateur fondamental des $\D$\nbd-équivalences) et \textbf{\boldmath CI\,$1'$g}.
\end{rem}

\begin{rem}
Dualement, si $\D$ est un prédérivateur conservatif et cocomplet, un foncteur entre petites catégories $u:A\toto B$ est $\D$\nbd-coasphérique si et seulement si le carré 
\[
\mathcal{D}\ =\quad
\raise 23pt
\vbox{
\UseTwocells
\xymatrixcolsep{2.5pc}
\xymatrixrowsep{2.4pc}
\xymatrix{
A\ar[d]_{p}\ar[r]^{u}
\drtwocell<\omit>{}
&B\ar[d]^{q}\ar@<-2ex>@{}[d]_(.55){=}
\\
e\ar[r]_{=}
&e
}
}
\]
satisfait à la propriété de changement de base homologique, autrement dit si le morphisme canonique $u^{}_!p^*\toto q^*$ est un isomorphisme. Si de plus le prédérivateur $\D$ admet aussi des images directes cohomologiques (en particulier si $\D$ est un dérivateur), cela équivaut à demander que le morphisme transposé $q^{}_*\toto p^{}_*u^*$ soit un isomorphisme, autrement dit, que le carré $\mathcal D$ satisfasse à la propriété de changement de base cohomologique. On est ainsi conduit à poser la définition suivante.
\end{rem}

\begin{df} \label{defcritcohcoasph}
Soit $\D$ un prédérivateur admettant des images directes cohomologiques. On dit que $\D$ \emph{satisfait le critère cohomologique pour les morphismes coasphériques} si pour toute flèche $\D$\nbd-coasphérique $u:A\toto B$ de $\Cat$, le carré
\[
\UseTwocells
\xymatrixcolsep{2.5pc}
\xymatrixrowsep{2.4pc}
\xymatrix{
A\ar[d]_{}\ar[r]^{u}
\drtwocell<\omit>{}
&B\ar[d]^{}\ar@<-2ex>@{}[d]_(.55){=}
\\
e\ar[r]_{}
&e
}
\]
satisfait à la propriété de changement de base cohomologique.
\end{df}

\begin{rem}
On ne connaît pas de prédérivateur conservatif et complet ne satisfaisant pas ce critère. Néanmoins, il ne semble pas être automatique si on ne suppose pas également l'existence d'images directes homologiques.
\end{rem}

\begin{thm} \label{dergchangbase}
Soit $\D$ un prédérivateur conservatif et complet satisfaisant le critère cohomologique pour les morphismes coasphériques \emph{(par exemple un dérivateur)}. Alors tout carré $\D$\nbd-exact satisfait à la propriété de changement de base cohomologique.
\end{thm}

\begin{proof}
Le théorème est conséquence directe du théorème~\ref{carclassex1}, dont les hypothèses sont satisfaites grâce à la proposition~\ref{sorchangbase}, (\emph{ii}), au corollaire~\ref{desclocchangbase}, au théorème~\ref{derfgchangbase}, (\emph{i}), et à la définition~\ref{defcritcohcoasph}.
\end{proof}

\begin{rem}
Contrairement aux foncteurs $\D$-asphériques, les morphismes de $\Cat$ qui sont des $\D$-équivalences localement sur une petite catégorie n'admettent pas une caractérisation en termes de carrés exacts. Néanmoins, on a une généralisation de l'équivalence des conditions~(\emph{a}) et~(\emph{d}) de la proposition~\ref{carfonctDasph}:
\end{rem}

\begin{prop} \label{carfonctDeqloc}
Soient $\D$ un prédérivateur conservatif et complet, et 
\[
\xymatrixcolsep{1pc}
\xymatrix{
A\ar[rr]^{u}\ar[dr]_{v}
&&B\ar[dl]^{w}
\\
&C
}
\]
un triangle commutatif de $\Cat$. Notons $p:A\toto e$, $q:B\toto e$ les flèches vers la catégorie ponctuelle. Les conditions suivantes sont équivalentes:
\begin{GMitemize}
\item[\rm a)] $u$ est une $\D$-équivalence localement sur $C$;
\item[\rm b)] le morphisme canonique $w^{}_*q^*\toto v^{}_*p^*$ est un isomorphisme.
\end{GMitemize}
\end{prop}

\begin{proof}
En vertu de l'axiome \textbf{Der\,2}, pour que le morphisme canonique
\[
w^{}_*q^*\toto w^{}_*u{}_*u^*q^*\simeq(wu)^{}_*(qu)^*=v^{}_*p^*
\]
soit un isomorphisme, il faut et il suffit que pour tout objet $c$ de $C$, la flèche
\begin{equation} \label{coucou}
c^*w^{}_*q^*\toto c^*v^{}_*p^*
\end{equation}
le soit. Or, en vertu de l'axiome \textbf{Der\,4g}, on a des isomorphismes canoniques
\[
\UseTwocells
\xymatrixcolsep{2.5pc}
\xymatrixrowsep{2.4pc}
\xymatrix{
\cm{B}{c}\ar[d]_{t}\ar[r]^{l}
\drtwocell<\omit>{}
&B\ar[d]^{w}
\\
e\ar[r]_{c}\ar@{}@<-4ex>[r]_{\textstyle c^*w^{}_*q^*\simeq t^{}_*l^*q^*=t^{}_*t^*\smsp,}
&C
}
\hskip 30pt
\xymatrix{
\cm{A}{c}\ar[d]_{s}\ar[r]^{k}
\drtwocell<\omit>{}
&A\ar[d]^{v}
\\
e\ar[r]_{c}\ar@{}@<-4ex>[r]_{\textstyle c^*v^{}_*p^*\simeq s^{}_*k^*p^*=s^{}_*s^*\smsp.}
&C
}
\]
On laisse au lecteur le soin de vérifier que ces isomorphismes identifient la flèche~\ref{coucou} au morphisme canonique $t^{}_*t^*\toto s^{}_*s^*$. Comme ce dernier est un isomorphisme si et seulement si le foncteur $\cm{A}{c}\toto\cm{B}{c}$ est une $\D$\nbd-équivalence (\cf~\ref{locfondder}), cela prouve la proposition.
\end{proof}

\begin{rem} \label{remcritcohcoloceq}
Dualement, si $\D$ est un prédérivateur conservatif et cocomplet, et
\[
\xymatrixcolsep{1pc}
\xymatrix{
A\ar[rr]^{u}\ar[dr]_{v}
&&B\ar[dl]^{w}
\\
&C
}
\]
un triangle commutatif dans $\Cat$, et si on note toujours $p:A\toto e$ et $q:B\toto e$ les flèches vers la catégorie ponctuelle, le foncteur $u$ est une $\D$\nbd-équivalence colocalement sur $C$ si et seulement si le morphisme canonique $v^{}_!p^*\toto w^{}_!q^*$ est un isomorphisme. Si de plus le prédérivateur $\D$ admet aussi des images directes cohomologiques (en particulier si $\D$ est un dérivateur), cela équivaut à demander que le morphisme transposé $q^{}_*w^*\toto p^{}_*v^*$ soit un isomorphisme. On est ainsi conduit à poser la définition suivante.
\end{rem}

\begin{df} \label{derpasfaible}
Soit $\D$ un prédérivateur admettant des images directes cohomologiques. On dit que $\D$ \emph{satisfait le critère cohomologique pour les équivalences colocales} si pour tout triangle commutatif dans $\Cat$
\[
\xymatrixcolsep{1pc}
\xymatrix{
A\ar[rr]^{u}\ar[dr]_{v}
&&B\ar[dl]^{w}
\\
&C
}
\]
tel que $u$ soit une $\D$\nbd-équivalence colocalement sur $C$, le morphisme canonique $q^{}_*w^*\toto p^{}_*v^*$, où $p:A\toto e$ et $q:B\toto e$ désignent les flèches vers la catégorie ponctuelle, est un isomorphisme. De façon équivalente, cette condition signifie que pour tout objet $X$ de $\D(e)$, et tout objet $Y$ de $\D(C)$, l'application 
\[
\Hom^{}_{\D(B)}(q^*(X),w^*(Y))\Toto{2}{}\Hom^{}_{\D(A)}(p^*(X),v^*(Y))\smsp,
\]
induite par $u^*$, est bijective. Un \emph{pseudo-dérivateur à gauche} est un pseudo-dérivateur faible à gauche satisfaisant le critère cohomologique pour les équivalences colocales, autrement dit, un prédérivateur conservatif et complet satisfaisant ce critère. Un \emph{dérivateur à gauche} est un pseudo-dérivateur à gauche satisfaisant de plus l'axiome de normalisation \textbf{Der\,1}, autrement dit, un dérivateur faible à gauche satisfaisant le critère cohomologique pour les équivalences colocales.
\end{df}

\begin{rem}
Soit $\D$ un prédérivateur admettant des images directes cohomologiques. Si $\D$ satisfait le critère cohomologique pour les équivalences colocales, il satisfait aussi, tautologiquement, le critère cohomologique pour les morphismes coasphériques. L'appellation de \og critère\fg{} pour ces deux conditions est justifiée par la proposition suivante.
\end{rem}

\begin{prop}
Soit $\D$ un prédérivateur conservatif et complet, et
\[
\xymatrixcolsep{1pc}
\xymatrix{
A\ar[rr]^{u}\ar[dr]_{v}
&&B\ar[dl]^{w}
\\
&C
}
\]
un triangle commutatif dans $\Cat$. Si le morphisme canonique $q^{}_*w^*\toto p^{}_*v^*$, où \hbox{$p:A\toto e$} et $q:B\toto e$ désignent les flèches vers la catégorie ponctuelle, est un isomorphisme, alors $u$ est une $\D$\nbd-équivalence colocalement sur $C$.
\end{prop}

\begin{proof}
Supposons que le morphisme canonique $q^{}_*w^*\toto p^{}_*v^*$ soit un isomorphisme. En vertu du théorème~\ref{derfgchangbase}, (\emph{i}), pour tout objet $c$ de $C$, les carrés comma
\[
\UseTwocells
\xymatrixcolsep{2.5pc}
\xymatrixrowsep{2.4pc}
\xymatrix{
\mc{c}{A}\ar[d]_{i}\ar[r]^-{s}
\drtwocell<\omit>{}
&e\ar[d]^{c}\ar@{}@<9ex>[d]^(.47){\hbox{et}}
\\
A\ar[r]_{v}
&C
}
\hskip 28pt
\xymatrix{
\mc{c}{B}\ar[d]_{j}\ar[r]^-{t}
\drtwocell<\omit>{}
&e\ar[d]^{c}
\\
B\ar[r]_{w}
&C
}
\]
satisfont à la propriété de changement de base cohomologique. On en déduit un diagramme commutatif d'isomorphismes
\[
\xymatrixrowsep{1.3pc}
\xymatrixcolsep{3.8pc}
\xymatrix{
&q^{}_*w^*c^{}_*\ar[r]^{\sim}\ar[d]^{\wr}
&p^{}_*v^*c^{}_*\ar[d]^{\wr}
\\
&q^{}_*j^{}_*t^*\ar[r]\ar[d]^{\wr}
&p^{}_*i^{}_*s^*\ar[d]^{\wr}
\\
&t^{}_*t^*\ar[r]
&s^{}_*s^*
&\hskip -60pt.
}
\]
On laisse au lecteur le soin de vérifier que la flèche horizontale du bas n'est autre que le morphisme canonique. Comme elle est un isomorphisme, on en déduit que le foncteur $\mc{c}{A}\toto\mc{c}{B}$ est une $\D$\nbd-équivalence (\cf~\ref{locfondder}), ce qui achève la démonstration.
\end{proof}

\begin{ex} \label{prederprefcatcmpl}
Soient $\C$ une catégorie, et $\D=\D_{\C}$ le prédérivateur des préfaisceaux à valeurs dans $\C$ (\cf~\ref{derpref}). Si la catégorie $\C$ est complète, alors $\D$ est un dérivateur à gauche. En effet, comme on a déjà vu que $\D$ est un dérivateur faible à gauche (\emph{loc. cit.}), il suffit de vérifier que $\D$ satisfait le critère cohomologique des équivalences colocales. Soit donc
\[
\xymatrixcolsep{1pc}
\xymatrix{
A\ar[rr]^{u}\ar[dr]_{v}
&&B\ar[dl]^{w}
\\
&C
}
\]
un triangle commutatif dans $\Cat$ tel que $u$ soit une $\D$-équivalence colocalement sur~$C$, et notons $p:A\toto e$, $q:B\toto e$ les foncteurs vers la catégorie ponctuelle. On veut montrer que la flèche canonique $q^{}_*w^*\toto p^{}_*v^*$ est un isomorphisme, ou de façon équivalente que pour tout objet $X$ de $\C$, et tout préfaisceau $Y$ sur $C$ à valeurs dans $\C$, l'application
\[
\Hom^{}_{\D(B)}(q^*(X),w^*(Y))\Toto{1.3}{}\Hom^{}_{\D(A)}(p^*(X),v^*(Y))\smsp,\quad\varphi\mapstoto u^*(\varphi)\smsp,
\]
est bijective. On va montrer que cette application est bijective pour $\C$ une catégorie \emph{arbitraire}, sans utiliser l'hypothèse qu'elle soit complète. On observe d'abord que dans le cas où elle est complète \emph{et} cocomplète, cela résulte des considérations de la remarque~\ref{remcritcohcoloceq}, ce qui en particulier règle le cas où $\C$ est équivalente à la catégorie ponctuelle. Ensuite, on montre qu'il existe un foncteur pleinement fidèle $i:\C\toto\C'$, avec $\C'$ catégorie complète et cocomplète telle que si $\D'$ désigne le dérivateur des préfaisceaux à valeurs dans $\C'$, on ait $\W_{\D'}=\W_{\D}$. Si $\C$ n'est pas un ensemble préordonné, on peut prendre pour $\C'$ la catégorie des préfaisceaux d'ensembles sur $\C$, et pour $i$ le plongement de Yoneda, puisque alors $\W_{\D'}=\W_{\D}=\Wzer$ (\cf~\ref{exlocfondder}). Si $\C$ est une catégorie, non vide et non équivalente à la catégorie ponctuelle, associée à un ensemble préordonné, alors on peut prendre pour $\C'$ la catégorie des préfaisceaux à valeurs dans la catégorie $\{0\toto1\}$, considérée comme sous-catégorie pleine de celle des ensembles, formée de l'ensemble vide et d'un ensemble ponctuel, et pour $i$ le foncteur induit par le plongement de Yoneda, puisque alors $\W_{\D'}=\W_{\D}=\Wgr$ (\cf~\ref{exlocfondder}). En observant alors que le morphisme $u$ est aussi une $\D'$\nbd-équivalence colocalement sur $C$, et en notant pour toute petite catégorie $K$, $i^{}_K:\D(K)\toto\D'(K)$ le foncteur pleinement fidèle induit par~$i$, on conclut en considérant le carré commutatif
\[
\xymatrixcolsep{3.5pc}
\xymatrixrowsep{.5pc}
\xymatrix{
\Hom^{}_{\D(B)}(q^*(X),w^*(Y))\ar[dd]_{\wr}\ar[r]
&\Hom^{}_{\D(A)}(p^*(X),v^*(Y))\ar[dd]^{\wr}
\\
\\
\Hom^{}_{\D(B)}(i^{}_Bq^*(X),i^{}_Bw^*(Y))\ar@{=}[d]
&\Hom^{}_{\D(A)}(i^{}_Ap^*(X),i^{}_Av^*(Y))\ar@{=}[d]
\\
\Hom^{}_{\D(B)}(q^*i^{}_e(X),w^*i^{}_C(Y))\ar[r]_{\sim}
&\Hom^{}_{\D(A)}(p^*i^{}_e(X),v^*i^{}_C(Y))\hskip 10pt,\hskip -15pt
}
\]
dont les flèches verticales, ainsi que la flèche horizontale du bas, sont des bijections.
\end{ex}

\begin{thm} \label{carcarexder} 
Soit $\D$ un prédérivateur conservatif et complet, satisfaisant le critère cohomologique pour les équivalences colocales \emph{(par exemple un dérivateur).} Alors un carré de $\Cat$ satisfait à la propriété de changement de base cohomologique si et seulement s'il est $\D$\nbd-exact faible.
\end{thm}

\begin{proof}
Soit
\[
\UseTwocells
\xymatrixcolsep{2.5pc}
\xymatrixrowsep{2.4pc}
\xymatrix{
A'\ar[d]_{u'}\ar[r]^{v}
\drtwocell<\omit>{\alpha}
&A\ar[d]^{u}
\\
B'\ar[r]_{w}
&B
}
\] 
un carré dans $\Cat$, et considérons le morphisme de changement de base $w^*u^{}_*\toto u'_*v^*$.
En vertu de \textbf{Der\,2}, pour qu'il soit un isomorphisme, il faut et il suffit que pour tout objet $b'$ de $B'$, la flèche 
\begin{equation} \label{glouglou}
{b'}^*w^*u^{}_*\toto {b'}^*u'_*v^*
\end{equation}
soit un isomorphisme. Or, en vertu de \textbf{Der\,4g}, on a des isomorphismes canoniques
\[
\UseTwocells
\xymatrixcolsep{2.5pc}
\xymatrixrowsep{2.4pc}
\xymatrix{
\cm{A}{w(b')}\ar[d]_{r}\ar[r]^-{j}
\drtwocell<\omit>{}
&A\ar[d]^{u}
\\
e\ar[r]_{w(b')}\ar@{}@<-4ex>[r]_{\textstyle {b'}^*w^*u^{}_*=w(b')^*u^{}_*\simeq r^{}_*j^*\smsp,}
&B
}
\hskip 30pt
\xymatrix{
\cm{A'}{b'}\ar[d]_{r'}\ar[r]^-{j'}
\drtwocell<\omit>{}
&A'\ar[d]^{u'}
\\
e\ar[r]_{b'}\ar@{}@<-4ex>[r]_{\textstyle {b'}^*u'_*v^*\simeq r'_*{j'}^*v^*=r'_*(vj')^*\smsp,}
&B'
}
\]
identifiant le morphisme~\ref{glouglou} à une flèche
\begin{equation} \label{guili}
r^{}_*j^*\toto r'_*(vj')^*\smsp.
\end{equation}
D'autre part, on a un carré commutatif
\[
\xymatrixcolsep{2.6pc}
\xymatrix{
&\cm{A'}{b'}\ar[d]_{j'}\ar[r]^{v'}
&\cm{A}{w(b')}\ar[d]^j
\\
&A'\ar[r]_v
&A
&\hskip -50pt,
}
\]
où $v'$ désigne le morphisme induit par $v$. On laisse au lecteur le soin de vérifier que le morphisme~\ref{guili} n'est autre que le morphisme canonique
\[
r^{}_*j^*\toto r^{}_*v'_*{v'}^*j^*\simeq(rv')^{}_*(jv')^*=r'_*(vj')^*\smsp,
\]
lequel, en vertu de l'hypothèse que $\D$ satisfait le critère cohomologique des équi\-va\-lences colocales, est un isomorphisme si et seulement si $v'$ est une $\D$\nbd-équivalence colocalement sur $A$. Le critère~(\emph{a}) de la proposition~\ref{carcarexf} permet alors de conclure. 
\end{proof}

\begin{rem} \label{excarBeckChev}
Soit
\[
\mathcal{D}\ =\quad
\raise 23pt
\vbox{
\UseTwocells
\xymatrixcolsep{2.5pc}
\xymatrixrowsep{2.4pc}
\xymatrix{
A'\ar[d]_{u'}\ar[r]^{v}
\drtwocell<\omit>{\alpha}
&A\ar[d]^{u}
\\
B'\ar[r]_{w}
&B
}
}\qquad
\]
un carré dans $\Cat$ tel que $u$ et $u'$ admettent des adjoints à droite $r$ et $r'$ respectivement, et $c:vr'\toto rw$ le morphisme \og de changement de base\fg. Pour tout prédérivateur $\D$, la $2$\nbd-fonctorialité implique que $(u^*,r^*)$ et $(u'{}^*,r'{}^*)$ sont des couples de foncteurs adjoints, de sorte que $u$ et $u'$ admettent des images directes cohomologiques $u^{}_*\simeq r^*$ et $u'_*\simeq r'{}^*$ respectivement, le morphisme de changement de base $w^*u^{}_*\toto u'_*v^*$ s'identifiant à $c^*:w^*r^*\toto r'{}^*v^*$. Si $\D$ est le dérivateur des préfaisceaux d'ensembles, et si le carré $\mathcal D$ est $\Wzer$\nbd-exact, il résulte donc du théorème précédent (\cf~exemple~\ref{exlocfondder}) que le morphisme $c^*$ est un isomorphisme. Le plongement de Yoneda étant un foncteur conservatif, on en déduit que le morphisme $c$ est lui-même un isomorphisme, autrement dit, que le carré $\mathcal D$ est de Beck-Chevalley à gauche, ce qui fournit une nouvelle preuve de l'implication \hbox{(\emph{c}) $\Rightarrow$ (\emph{a})} de la proposition~\ref{carcarBeckChev}. Le lecteur vérifiera l'absence de cercle vicieux.
\end{rem}

\begin{cor}
Soit
\[
\mathcal{D}\ =\quad
\raise 23pt
\vbox{
\UseTwocells
\xymatrixcolsep{2.5pc}
\xymatrixrowsep{2.4pc}
\xymatrix{
A'\ar[d]_{u'}\ar[r]^{v}
\drtwocell<\omit>{\alpha}
&A\ar[d]^{u}
\\
B'\ar[r]_{w}
&B
}
}\qquad
\]
un carré dans $\Cat$. Les conditions suivantes sont équivalentes:
\begin{GMitemize}
\item[\rm a)] $\mathcal D$ est un carré exact de Guitart, autrement dit un carré $\Wzer$\nbd-exact;
\item[\rm b)] le carré
\[
\pref{\mathcal{D}}\ =\quad
\raise 25pt
\vbox{
\UseTwocells
\xymatrixcolsep{2.5pc}
\xymatrixrowsep{2.4pc}
\xymatrix{
\pref{B}\ar[d]_{u^*}\ar[r]^{w^*}
\drtwocell<\omit>{\alpha^*}
&\pref{B}'\ar[d]^{u'{}^*}
\\
\pref{A}\ar[r]_{v^*}
&\pref{A}'
}
}\qquad
\]
\emph{(où pour toute petite catégorie $C$, $\pref{C}$ désigne la catégorie des préfaisceaux d'ensembles sur $C$)} est un carré exact de Guitart;
\item[\rm c)] pour toute catégorie complète ou cocomplète $\C$, le carré
\[
\Homint{(\op{\mathcal{D}},\C)}\ =\quad
\raise 28pt
\vbox{
\UseTwocells
\xymatrixcolsep{2.5pc}
\xymatrixrowsep{3.4pc}
\xymatrix{
\Homint{(\op{B},\C)}\ar[d]_{u^*}\ar[r]^{w^*}
\drtwocell<\omit>{\alpha^*}
&\Homint(\op{B'},\C)\ar[d]^{u'{}^*}
\\
\Homint{(\op{A},\C)}\ar[r]_{v^*}
&\Homint(\op{A'},\C)
}
}\qquad
\]
\emph{(où pour toute petite catégorie $C$, $\Homint{(\op{C},\C)}$ désigne la catégorie des préfaisceaux sur $C$ à valeurs dans $\C$)} est un carré exact de Guitart;
\item[\rm d)] il existe une catégorie complète ou cocomplète $\C$, qui ne soit pas un ensemble préordonné, et telle que $\Homint{(\op{\mathcal{D}},\C)}$ soit un carré exact de Guitart.
\end{GMitemize}  
\end{cor}

\begin{proof}
Les implications (\emph{c}) $\Rightarrow$ (\emph{b}) et (\emph{b}) $\Rightarrow$ (\emph{d}) sont tautologiques. En vertu des exemples~\ref{exlocfondder} et~\ref{prederprefcatcmpl}, et des propositions~\ref{carcarBeckChev} et~\ref{carcarexf}, les implications  \hbox{(\emph{a}) $\Rightarrow$ (\emph{c})} et (\emph{d}) $\Rightarrow$ (\emph{a}) résultent, dans le cas complet, du théorème~\ref{carcarexder}. Dans le cas cocomplet, elles se démontrent de façon duale.
\end{proof}

\begin{rem}
En gardant les hypothèses et les notations du corollaire précédent, si $\C$ n'est pas une catégorie complète ou cocomplète, il n'est pas vrai en général que pour un carré $\Wzer$\nbd-exact $\mathcal D$, le carré $\Homint(\op{\mathcal D},\C)$ soit exact au sens de Guitart, et cela même si l'on suppose que $\mathcal D$ est un carré comma. Voici un contre-exemple: Soient $A$ et $B$ deux petites catégories, et considérons le carré 
\[
{\mathcal{D}}\ =\quad
\raise 23pt
\vbox{
\UseTwocells
\xymatrixcolsep{2.5pc}
\xymatrixrowsep{2.4pc}
\xymatrix{
A\times B\ar[d]_{pr^{}_1}\ar[r]^-{pr^{}_2}
\drtwocell<\omit>{}
&B\ar[d]^{}\ar@<-2.7ex>@{}[d]_(.55){=}
\\
A\ar[r]_{}
&e\hskip 15pt,\hskip -18pt
}
}
\]
qui est à la fois un carré cartésien et un carré comma. Pour une catégorie $\C$, dire que le carré 
\[
\Homint{(\op{\mathcal{D}},\C)}\ =\hskip 20pt
\raise 28pt
\vbox{
\UseTwocells
\xymatrixcolsep{2.5pc}
\xymatrixrowsep{3.4pc}
\xymatrix{
\hskip -19pt\C\simeq\Homint{(\op{e},\C)}\ar[d]_{}\ar[r]^{}
\drtwocell<\omit>{}
&\Homint(\op{A},\C)\ar[d]^{pr_1^*}\ar@<-8.3ex>@{}[d]_(.55){=}
\\
\Homint{(\op{B},\C)}\ar[r]_-{pr_2^*}
&\Homint(\op{A}\times\op{B},\C)
}
}\qquad
\]
est un carré exact de Guitart, revient à demander, en vertu du critère (\emph{c}) de la proposition~\ref{carcarex}, que pour tous préfaisceaux $F$ sur $A$ et $G$ sur $B$, à valeurs dans $\C$, et tout morphisme $\varphi:pr^{*}_1(F)\toto pr^{*}_2(G)$ de préfaisceaux sur $A\times B$, la catégorie $\C_{F,G,\varphi}$ (dont les objets sont les triplets $(c,\beta,\gamma)$, avec $c$ objet de $C$, $\beta$ morphisme de pré\-fai\-sceaux de $F$ vers le préfaisceau constant sur $A$, de valeur $c$, et $\gamma$ morphisme du pré\-fai\-sceau constant sur $B$, de valeur $c$, vers le préfaisceau $G$, tels que \hbox{$pr_2^*(\gamma)pr_1^*(\beta)=\varphi$}, un morphisme de $(c,\beta,\gamma)$ vers $(c',\beta',\gamma')$ étant une flèche $f:c\toto c'$ telle que $\beta'=f\beta$ et $\gamma=\gamma'f\kern 0pt$) est 0\nbd-connexe. Or, si $A=\{a_0,a_1\}$ et $B=\{b_0,b_1\}$ sont des catégories discrètes à deux objets, $\C$ la catégorie
\[
\C=\left\{
\raise 19pt
\vbox{
\xymatrixcolsep{2.9pc}
\xymatrixrowsep{1.3pc}
\xymatrix{
a_0\ar[r]\ar[rd]
&b_0
\\
a_1\ar[r]\ar[ru]
&b_1
}
}
\right\}\quad,
\]
$F:\op{A}=A\toto\C$ et $G:\op{B}=B\toto\C$ les inclusions évidentes, et $\varphi$ l'unique morphisme de préfaisceaux sur $A\times B$ de $pr^{*}_1(F)$ vers $pr^{*}_2(G)$, alors on vérifie aussitôt que la catégorie $\C_{F,G,\varphi}$ est vide, et donc n'est pas 0\nbd-connexe. On en déduit que dans ce cas  $\Homint{(\op{\mathcal{D}},\C)}$ n'est \emph{pas} un carré exact de Guitart (et même pas un carré $\Wgr$\nbd-exact).
\end{rem}

\backmatter

\bibliographystyle{smfplain}

\begin{thebibliography}{10}

\bibitem{SGA4}
{\scshape M.~Artin{, A. Grothendieck, J.-L. Verdier}} -- \emph{Th\'eorie des
  topos et cohomologie \'etale des sch\'emas \emph{(SGA4)}}, Lecture Notes in
  Mathematics, Vol. 269, 270, 305, Springer-Verlag, 1972-1973.

\bibitem{CiDer}
{\scshape D.-C. Cisinski} -- {\og Images directes cohomologiques dans les
  cat\'egories de mod\`eles\fg}, \emph{Annales Math\'ematiques Blaise Pascal}
  \textbf{10} (2003), p.~195--244.

\bibitem{Ci2}
\bysame , {\og Le localisateur fondamental minimal\fg}, \emph{Cahiers Topologie
  G\'eom. Dif\-f\'e\-ren\-tielle Cat\'eg.} \textbf{XLV (2)} (2004),
  p.~109--140.

\bibitem{CiAst}
\bysame , \emph{Les pr\'efaisceaux comme mod\`eles des types d'homotopie},
  Ast\'erisque, Vol. 308, Soc. Math. France, 2006.

\bibitem{CiCatDer}
\bysame , {\og Cat\'egories d\'erivables\fg}, \emph{Bull. Soc. Math. France}
  \textbf{138} (2010), p.~317--393.

\bibitem{CiGM}
{\scshape D.-C. Cisinski {\normalfont \smfandname} G.~Maltsiniotis} -- {\og La
  cat\'egorie {$\mathbf\Theta$} de {J}oyal est une cat\'egorie test\fg},
  \emph{J. Pure Appl. Algebra} \textbf{215} (2011), p.~962--982.

\bibitem{DHKS}
{\scshape W.~G. Dwyer, P.~S. Hirschhorn, D.~M. Kan {\normalfont \smfandname}
  J.~H. Smith} -- \emph{Homotopy limit functors on model categories and
  homotopical categories}, Mathematical Surveys and Monographs, Vol. 113,
  American Mathematical Society, 2004.

\bibitem{Franke}
{\scshape J.~Franke} -- {\og Uniqueness theorems for certain triangulated
  categories possessing an {A}dams spectral sequence\fg}, K-theory Preprint
  Archives, 139, 1996.

\bibitem{SGA1}
{\scshape A.~Grothendieck} -- \emph{Rev\^etements \'etales et groupe
  fondamental {{\rm (SGA1)}}}, Lecture Notes in Mathematics, Vol. 224,
  Springer-Verlag, 1971.

\bibitem{PS}
\bysame , {\og \emph{Pursuing stacks}\fg}, Manuscrit, 1983, \`a para\^\i tre
  dans \emph{Documents Ma\-th\'e\-ma\-tiques}.

\bibitem{Der}
\bysame , {\og \emph{Les d\'erivateurs}\fg}, Manuscrit,
  1990,~www.math.jussieu.fr/\raise -3.3pt\vbox{\hbox{$\widetilde{ \
  }\,$}}maltsin/groth \hbox{/Derivateurs.html}.

\bibitem{Guit1}
{\scshape R.~Guitart} -- {\og Relations et carr\'es exacts\fg}, \emph{Ann. sc.
  math. Qu\'ebec} \textbf{IV, 2} (1980), p.~103--125.

\bibitem{Guit2}
\bysame , {\og Carrés exacts et carrés déductifs\fg}, \emph{Diagrammes}
  \textbf{6} (1981), p.~G1--G17.

\bibitem{GuitSplit}
\bysame , {\og \emph{Split exact squares and torsion}\fg}, Preprint, 1988.

\bibitem{Guit3}
{\scshape R.~Guitart {\normalfont \smfandname} L.~Van~den Bril} -- {\og Note
  sur la détermination des homologies par les carrés exacts\fg},
  \emph{Diagrammes} \textbf{6} (1981), p.~GV1--GV7.

\bibitem{Guit4}
\bysame , {\og Calcul des satellites et présentations des bimodules à l'aide
  des carrés exacts\fg}, \emph{Cahiers Topologie Géom. Différentielle}
  \textbf{XXIV, 3} (1983), p.~299--330.

\bibitem{Guit5}
\bysame , {\og Calcul des satellites et présentations des bimodules à l'aide
  des carrés exacts,~{II}\fg}, \emph{Cahiers Topologie Géom. Différentielle}
  \textbf{XXIV, 4} (1983), p.~333--369.

\bibitem{Hel}
{\scshape A.~Heller} -- {\og Homotopy theories\fg}, \emph{Mem. Amer. Math.
  Soc.} \textbf{71} (1988), no.~383.

\bibitem{Hel2}
\bysame , {\og Stable homotopy theories and stabilization\fg}, \emph{J. Pure
  Appl. Algebra} \textbf{115} (\gmremarque{b}1997), p.~113--130.

\bibitem{Hel3}
\bysame , {\og Homological algebra and (semi)stable homotopy\fg}, \emph{J. Pure
  Appl. Algebra} \textbf{115} (\gmremarque{c}1997), p.~131--139.

\bibitem{Hel4}
\bysame , {\og Semistability and infinite loop spaces\fg}, \emph{J. Pure Appl.
  Algebra} \textbf{154} (\gmremarque{d}2000), p.~213--220.

\bibitem{BKGM}
{\scshape B.~Kahn {\normalfont \smfandname} G.~Maltsiniotis} -- {\og Structures
  de d\'erivabilit\'e\fg}, \emph{Adv. in Math.} \textbf{218} (2008),
  p.~1286--1318.

\bibitem{Bernhard}
{\scshape B.~Keller} -- {\og Derived categories and universal problems\fg},
  \emph{Comm. Algebra} \textbf{19} (1991), p.~699--747.

\bibitem{Mal1}
{\scshape G.~Maltsiniotis} -- {\og Introduction à la th\'eorie des
  d\'erivateurs\fg}, Pr\'epublication, 2001, www.math.jussieu.fr/\raise
  -3.3pt\vbox{\hbox{$\widetilde{ \ }\,$}}maltsin/.

\bibitem{Ast}
\bysame , \emph{La th\'eorie de l'homotopie de {G}rothendieck}, Ast\'erisque,
  Vol.~301, Soc. Math. France, 2005.

\bibitem{Mal2}
\bysame , {\og Structures d'asph\'ericit\'e, foncteurs lisses, et
  fibrations\fg}, \emph{Annales Ma\-th\'e\-ma\-tiques Blaise Pascal}
  \textbf{12} (2005), p.~1--39.

\bibitem{Qu0}
{\scshape D.~Quillen} -- \emph{Homotopical algebra}, Lecture Notes in
  Mathematics, Vol. 43, Springer-Verlag, 1967.

\bibitem{Qu}
\bysame , {\og Higher algebraic {K}-theory: I\fg}, \emph{Algebraic K-theory I},
  Lecture Notes in Mathematics, Vol. 341, p.~85-147, Springer-Verlag, 1973.

\end{thebibliography}

\providecommand{\bysame}{\leavevmode ---\ }
\providecommand{\og}{``}
\providecommand{\fg}{''}
\providecommand{\smfandname}{et}
\providecommand{\smfedsname}{\'eds.}
\providecommand{\smfedname}{\'ed.}
\providecommand{\smfmastersthesisname}{M\'emoire}
\providecommand{\smfphdthesisname}{Th\`ese}

\end{document}